\documentclass[11pt]{article}
\usepackage[margin = 1in]{geometry}

\usepackage[affil-it]{authblk}
\usepackage{lipsum}
\usepackage{amsfonts}
\usepackage{graphicx}
\graphicspath{ {images/} }
\usepackage{epstopdf}
\usepackage{dsfont}
\usepackage{mathtools}
\usepackage{empheq}
\usepackage{amsfonts}
\usepackage{amsgen,amsthm,amsmath,amstext,amsbsy,amsopn,amssymb,subcaption,stmaryrd,mathabx,mathrsfs}
\usepackage{comment}
\usepackage[font=footnotesize,labelfont=bf]{caption}

\usepackage{blkarray}

\usepackage{url}
\usepackage{longtable}
\usepackage{mathtools}
\usepackage{multirow}
\usepackage{array}

\usepackage{enumitem}
\usepackage{hyperref}
\usepackage{natbib}
\usepackage{booktabs}
\usepackage{algorithm}
\usepackage{algpseudocode}
\algrenewcommand\algorithmiccomment[1]{\hfill\texttt{//~#1}}
\algrenewcommand{\algorithmicrequire}{\textbf{Input:}}
\algrenewcommand{\algorithmicensure}{\textbf{Output:}}
\usepackage{xcolor}
\usepackage[utf8]{inputenc} 
\usepackage[T1]{fontenc}

\ifpdf
  \DeclareGraphicsExtensions{.eps,.pdf,.png,.jpg}
\else
  \DeclareGraphicsExtensions{.eps}
\fi

\newtheorem{Theorem}{Theorem}

\newtheorem{Lemma}{Lemma}
\newtheorem{Corollary}{Corollary}

\newtheorem{Proposition}{Proposition}

\newtheorem{Claim}{Claim}

\theoremstyle{definition}

\DeclareMathAlphabet\mathbfcal{OMS}{cmsy}{b}{n}

\usepackage{bbm}

\newcommand{\cD}{\mathcal{D}}
\newcommand{\cE}{\mathcal{E}}
\newcommand{\cF}{\mathcal{F}}

\newcommand{\cN}{\mathcal{N}}

\newcommand{\cT}{\mathcal{T}}

\newcommand{\cX}{\mathcal{X}}

\newcommand{\bbN}{\mathbb{N}}

\newcommand{\bbR}{\mathbb{R}}

\newcommand{\bbZ}{\mathbb{Z}}

\DeclarePairedDelimiter\ceil{\lceil}{\rceil}
\DeclarePairedDelimiter\floor{\lfloor}{\rfloor}
\DeclarePairedDelimiter\abs{\lvert}{\rvert}
\DeclarePairedDelimiter\norm{\lVert}{\rVert}

\newcommand{\dif}{\,\mathrm{d}}
\newcommand{\im}{\mathrm{i}}
\newcommand{\iid}{\overset{\mathrm{iid.}}{\sim}}
\newcommand{\define}{\overset{\mathrm{def}}{=}}

\newcommand{\HL}{\mathsf{H}}

\DeclareMathOperator{\E}{\mathbb{E}} 
\DeclareMathOperator{\Prob}{\mathbb{P}} 
\DeclareMathOperator{\Cov}{\operatorname{Cov}} 
\DeclareMathOperator{\Var}{\operatorname{Var}} 
\newcommand{\indi}{{\mathds{1}}}

\DeclareMathOperator*{\argmin}{\arg\min}
\DeclareMathOperator*{\argmax}{\arg\max}

\newcommand{\eqd}{\overset{\mathrm{d}}{=}}

\hypersetup{
    colorlinks=true,
    linkcolor=blue,
    filecolor=blue,      
    urlcolor=cyan,
    citecolor=blue
}

\makeatletter
\newcommand*{\rom}[1]{\expandafter\@slowromancap\romannumeral #1@}
\makeatother

\allowdisplaybreaks

\begin{document}
	\title{Instance-Optimal Adaptive Location Estimation via\\ Multiscale Mid-Summaries}
	\author{Qiaosen Wang}
    \author{Chao Gao\thanks{The research of QW and CG is supported in part by NSF Grants ECCS-2216912 and DMS-2310769.}}
    \affil{Department of Statistics, University of Chicago}
	\date{\today}
	\maketitle

\begin{sloppypar}
\begin{abstract}
Location estimation exhibits markedly different finite-sample behavior across noise distributions: regular families typically yield root-\(n\) rates, whereas compactly supported laws may admit faster, boundary-driven rates. We question whether a single estimator, without knowledge of the density's shape, can adapt to the instance-wise optimal estimation rate, as an oracle that knows the underlying location family can. 

For a known location family with symmetric log-concave noise density \(f\), the optimal location estimation error with sample size \(n\) under failure probability \(\delta\) is known to be Le Cam's two-point rate:
\[ \sup\left\{r>0:\mathsf{H}^2\left(f_0, f_{2r}\right)\lesssim \frac{\log(1/\delta)}{n}\right\}.
\]
When the location family is unknown, we propose a shape-agnostic estimator that attains this oracle benchmark simultaneously over all symmetric unimodal densities with non-decreasing hazard rates, a class strictly broader than symmetric log-concave distributions. We establish that the Hellinger-driven two-point rate can be characterized solely by a multiscale function of dyadic quantile gaps. This new structural connection between Hellinger divergence and quantile geometry motivates a simple estimation procedure that aggregates sample mid-summaries with carefully designed data-dependent weights. The resulting estimator is finite-sample instance-optimal and runs in only $O(\log(n))$ time on sorted samples.

\end{abstract}
\tableofcontents

\section{Introduction}\label{sec:intro}

Location estimation is one of the most classical problems in statistics. One observes
\begin{align}
    X_i\iid  f_\mu\define f(\cdot-\mu),\quad i=1,2,\cdots,n,
\end{align}
where $\mu\in\bbR$ is the location parameter we want to estimate, and $f(\cdot)$ is a probability density that serves as the noise profile. It is typically assumed that $f$ is symmetric and unimodal, so that the location $\mu$ can be meaningfully and uniquely determined from the sampling distribution. Fix a target accuracy level $r>0$ and a failure probability $\delta\in(0,1)$. The central goal is to determine how many samples are needed to estimate $\mu$ within error $r$ with probability at least $1-\delta$.

Despite its apparent simplicity, the information-theoretic hardness of this problem can vary dramatically with the shape of $f$. In regular location families like Gaussian and Laplace, the optimal estimation error has the familiar root-$n$ order. In contrast, compactly supported families, including uniform, semicircle, and triangle distributions, allow for faster and boundary-driven convergence rates; therefore, substantially fewer samples are needed at the same level of accuracy. Thus, there is no single distribution-free rate that captures the finite-sample difficulty of location estimation.

The most elementary difficulty in location estimation is distinguishing between locations separated by a distance of $2r$. According to Le Cam's two-point argument, a lower bound for estimation can be obtained from the detection threshold for the following pair of hypotheses:
\begin{align}\label{eq:two-point testing}
    H_0:\ X_{1:n}\iid f_0\qquad \textnormal{versus}\qquad H_1(r):\ X_{1:n}\iid f_{2r}.
\end{align}
Indistinguishability between $H_0$ and $H_1(r)$, namely $H^2(f_0,f_{2r})\lesssim_\delta 1/n$, implies a lower bound of $r$ for all measurable estimators. Therefore, given any failure probability $\delta\in(0,1/3]$, any estimator $\widehat{\mu}$ (whether it has access to $f$ or not) with $n\geq C'_0\log(1/\delta)$ samples must satisfy
\begin{align}\label{eq:lower_bound_known_f}
    \inf_{\mu\in\bbR}\,\Prob_{X_{1:n}\iid f_\mu}\left(\abs{\widehat{\mu}-\mu}> \omega_f\left(\frac{C_0\log(1/\delta)}{n}\right)\right)\geq  \delta
\end{align}
for universal absolute constants $C_0$ and $C'_0$, where
\begin{align}\label{eq:inverse_h2_modulus_of_continuity}
    \omega_f\left(t\right)\define \sup\Big\{r\geq 0:\ \HL^2\left(f_0,f_{2r}\right)\leq t\Big\},\quad t\in [0,1],
\end{align}
is the \emph{inverse Hellinger modulus of continuity} of the location family $\{f_\mu:\, \mu\in\bbR\}$. When the shape of the density is known, the benchmark \eqref{eq:lower_bound_known_f} is also attainable up to universal constants for every $f$ belonging to the class of symmetric and log-concave (thus unimodal) densities:
\begin{align}\label{eq:slc}
        \cF_\mathsf{SLC}\define \Big\{f\ \textnormal{is a probability density}:\, f(x)=f(-x),\ \log f:\bbR\to [-\infty,+\infty)\, \textnormal{is concave}\Big\}.
\end{align}
Indeed, the parametric maximum likelihood estimator (MLE)
\begin{align}\label{eq:mle}
    \widehat{\mu}_f^{(\mathsf{MLE})}\in \argmax_{\theta\in\bbR} \sum_{i=1}^n \log f(X_i-\theta)
\end{align}
can achieve the lower bound \eqref{eq:lower_bound_known_f} for every $f\in\cF_\mathsf{SLC}$ up to a universal constant (see Claim~\ref{clm:mle} of Appendix~\ref{sec:proof_mle}). Consequently, the minimax-optimal estimation error $r=r(n,\delta;f)$ given $n\gtrsim \log(1/\delta)$ samples is fully characterized by the critical equation:
\begin{align}\label{eq:critical_equation_intro}
    \HL^2\left(f_0,f_{2r}\right)\,\asymp\,\frac{\log(1/\delta)}{n}\quad \Longleftrightarrow\quad r=\omega_f\Big(\Theta\Big(\frac{\log(1/\delta}{n}\Big)\Big)
\end{align}
for every known $f$ in the class of symmetric log-concave distributions. For example:
\begin{itemize}
    \item when $f$ is $\cN(0,1)$ or $\mathrm{Laplace}(0,1)$, $\omega_f(t)\asymp \sqrt{t}$, and the optimal rate is familiarly $\Theta_\delta(1/\sqrt{n})$; 
    \item when $f$ is $\mathrm{Unif}[-1,1]$, $\omega_f(t)\asymp t$, and the optimal rate is $\Theta_\delta(1/n)$;
    \item when $f(x)=(2/\pi)\sqrt{(1-x^2)_+}$ is the semicircle density,  $\omega_f(t)\asymp t^{2/3}$, and the optimal rate is $\Theta_\delta(n^{-2/3})$;
    \item when $f(x)=(1-\abs{x})_+$ is the triangle density or $f(x)=(3/4)(1-x^2)_+$ is the Epanechnikov (parabolic) kernel, $\omega_f(t)\asymp \sqrt{t/\log(1/t)}$, and the optimal rate is $\Theta_\delta(1/\sqrt{n\log(n)})$.
\end{itemize}

Nevertheless, this location estimation task becomes intrinsically semiparametric and genuinely challenging when $f$ is unknown. The optimality benchmark \eqref{eq:critical_equation_intro} itself depends on the unknown density, and different shapes call for very different estimators. The known-$f$ MLE \eqref{eq:mle} is optimal for its corresponding $f\in\cF_\mathsf{SLC}$ but is not universally optimal for every other $f\in\cF_\mathsf{SLC}$. For example, the sample mean is a minimax-optimal MLE for the Gaussian location family but cannot exploit the sharp support boundary to achieve the fast $\Theta(1/n)$ rate for uniform distributions. Conversely, the sample mid-range is unimprovable for the uniform location family but has a barely satisfactory $\Theta(1/\sqrt{\log(n)})$ error when the underlying density $f$ is secretly Gaussian. No single fixed order statistic or classical location estimator is naturally tuned to all of these diverse density shapes. This leads to the central question of the paper:
\begin{quote}
\itshape
\quad Can a single location estimator, without knowing the density $f$, adaptively attain the instance-wise optimal rate given by \eqref{eq:critical_equation_intro} as an oracle that knows $f$?
\end{quote}

To answer this question, one needs to posit $f$ in a class that is broad enough to contain genuinely different estimation regimes yet structured enough that its local Hellinger geometry can be learned from the data. We start with the class $\cF_\mathsf{SLC}$ of symmetric log-concave densities, as defined previously by \eqref{eq:slc}. This class encompasses both regular distributions, such as Gaussian, Laplace, and logistic, and compactly supported distributions, such as uniform, semicircle, triangle, and Epanechnikov. At the same time, log-concavity prevents narrow hidden spikes and oscillations that could create substantial location information without being reflected in the observable data. It therefore provides a proper setting in which adaptation is both nontrivial and statistically possible.

Our main theorem answers this question affirmatively for all unknown symmetric log-concave densities and beyond:
\begin{Theorem}\label{thm:main}
For any $\delta\in(0,1/3]$, there exists an estimator $\widehat{\mu}_\delta=\widehat{\mu}_\delta(X_{1:n})$ that is agnostic of $f$, such that 
\begin{align*}
    \inf_{\mu\in\bbR}\ \Prob_{X_{1:n}\iid f_\mu}\left(\abs{\widehat{\mu}_\delta-\mu}\leq \omega_f\Big(\frac{C\log(1/\delta)}{n}\Big)\right)\geq 1-\delta\quad\textnormal{whenever}\  n\geq C'\log(1/\delta),
\end{align*}
for any $f\in\cF_\mathsf{SLC}$, where $C$ and $C'$ are universal absolute constants. Moreover, the same conclusion can be extended to every unknown symmetric unimodal density $f$ that has a monotone hazard rate.
\end{Theorem}

The main results of this paper are as follows:
\paragraph{Cost-free adaptation to symmetric log-concave tails and beyond.}Our Theorem~\ref{thm:main} reveals the feasibility of adapting to unknown symmetric log-concave tails and achieving precise instance-optimality in one-dimensional location estimation. The adaptive estimation procedure we propose is the first to incur no adaptation cost other than an enlargement of the universal constants. The best preceding result by \cite{compton2026attainability} uses interval-frequency comparisons based on a reverse data-processing inequality from \cite{pensia2023communication}; their result is nearly free of adaptation costs over $\cF_\mathsf{SLC}$, but it is still loose by $\operatorname{polylog}(n,1/\delta)$ factors. We show that the binary reduction based on the reverse data-processing inequality can incur an unavoidable logarithmic information loss. Our approach instead builds on a novel multiscale characterization of the two-point rate through dyadic quantile gaps, which guides the construction of a data-adaptively weighted combination of sample mid-summaries. This new perspective enables both cost-free adaptation and extension beyond log-concavity to symmetric unimodal densities with monotone hazard rates.

\paragraph{Recovering Le Cam's two-point rate from multiscale quantile geometry.}The key ingredient for our cost-free adaptation is a new structural connection between multiscale quantile geometry and Le Cam's two-point testing lower bound in the location estimation setting. We identify a crucial function
\begin{align}\label{eq:delta_intro}
    \Delta_f(p)\define \left(\sum_{j\in\bbN:\,2^jp\leq 1/4}\frac{2^j}{\left(q(2^{j+1}p)-q(2^jp)\right)^2}\right)^{-1/2},
\end{align}
where $p\in(0,1/4]$ is a percentile and $q(\cdot)$ is the quantile function of $f$. Our analysis establishes that
\begin{align}\label{eq:h2_delta_intro}
    \HL^2\left(f_0,f_{\Theta(1)\cdot \Delta_f(\frac{\Theta(1)\cdot\log(1/\delta)}{n})}\right)\asymp \frac{\log(1/\delta)}{n}
\end{align}
for every symmetric unimodal density $f$ with a monotone hazard rate (MHR), which can be implied by log-concavity. At a high level, the identity \eqref{eq:h2_delta_intro} arises from a piecewise constant approximation to the Hellinger integral, where the approximation tightness is ensured by the property that every MHR density is almost flat between two consecutive dyadic quantiles. Given \eqref{eq:h2_delta_intro}, the Hellinger-based two-point rate in \eqref{eq:critical_equation_intro}, namely the minimax estimation error in the known-$f$ setting, can be represented by:
\begin{align}\label{eq:inverse_modulus_delta_intro}
    \omega_f\left(\Theta\Big(\frac{\log(1/\delta)}{n}\Big)\right)= \Delta_f\left(\Theta\Big(\frac{\log(1/\delta)}{n}\Big)\right).
\end{align}
As this universal identity indicates, a rate of $\Delta_f(\Theta(\log(1/\delta)/n))$ suffices for instance-optimality. This alternative formulation of the optimal rate is more tractable than the inverse Hellinger modulus form since the functional $\Delta_f(\cdot)$ depends only on dyadic quantile gaps and can thus be empirically recovered from sample quantiles.

\paragraph{Revitalizing the simple yet efficient art of sample mid-summaries.} Motivated by the aforementioned Hellinger-quantile connection, we design a simple adaptive estimator $\widehat{\mu}_\delta$ based on order statistics, particularly the sample mid-summaries:
\begin{align*}
    M_k(X_{1:n})\define \frac{X_{(k)}+X_{(n+1-k)}}{2},\quad k=1,2,\cdots,\ceil{n/2}.
\end{align*}
Note that several parametric MLEs, including the sample mean (optimal for Gaussian), the sample median (optimal for Laplace), and the sample mid-range (optimal for uniform), can all be written as a convex combination of sample mid-summaries. However, they all use fixed weights to combine the mid-summaries and cannot adapt as the underlying noise distribution $f$ changes. To overcome this limitation, we propose using data-dependent weights to aggregate the sample mid-summaries at dyadic indices, i.e., estimating $\mu$ with
\begin{align*}
    \widehat{\mu}_\delta =\sum_{j\in\bbN:\, 2^j (k_\delta/n)\leq 1/4}\widehat{\lambda}_j\cdot M_{2^jk_\delta},
\end{align*}
where $k_\delta\asymp \log(1/\delta)$ and $\{\widehat{\lambda}_j\}$ are data-dependent convex weights selected using a hold-out set of samples. We only use dyadic sample mid-summaries for three reasons: first, it matches the form of the critical quantity \eqref{eq:delta_intro} that is linked to the optimal rate; second, two sample quantiles are already statistically similar if their corresponding percentiles are constant multiples of each other, and hence there is no need to keep track of all sample quantiles between them; third, as we will soon see, dyadic quantiles are almost uncorrelated, which facilitates the concentration analysis. Specifically, we observe that
\begin{align*}
    \Var\left(M_{2^jk_\delta}\right)\approx \frac{\left(q(2^{j+1}k_\delta/n)-q(2^jk_\delta/n)\right)^2}{2^j},\quad \operatorname{Corr}\left(M_{2^jk_\delta},M_{2^lk_\delta}\right)\approx 2^{-\abs{j-l}/2}.
\end{align*}
These two properties determine the optimal choice of weights
\begin{align*}
    \widehat{\lambda}_j= \frac{\Delta_f^2(k_\delta/n)\cdot2^j}{\left(q(2^{j+1}k_\delta/n)-q(2^jk_\delta/n)\right)^2}
\end{align*}
and the optimized total error
\begin{align*}
    \left(\sum_{j}\widehat{\lambda}_j^2\cdot \Var\left(M_{2^jk_\delta}\right)\right)^{1/2}\asymp  \Delta_f(k_\delta/n),
\end{align*}
which appears to be the optimal rate according to \eqref{eq:inverse_modulus_delta_intro}. With this procedure, we adaptively attain the instance-optimal rate \eqref{eq:critical_equation_intro} in $\mathrm{polylog}(n)$ extra time modulo a standard $O(n\log(n))$ time complexity for sorting the samples.

\subsection{Notation}
We use standard big-$O$ asymptotic notations $O(\cdot)$ ($\lesssim$), $\Omega(\cdot)$ ($\gtrsim$), $\Theta(\cdot)$ ($\asymp$). For any two distributions $P$ and $Q$ on the same probability space $\cX$, we define
\begin{align*}
    \HL^2(P,Q)\define \frac{1}{2}\int_\cX \left(\sqrt{\dif P}-\sqrt{\dif Q}\right)^2=1-\int_\cX \sqrt{\dif P\dif Q}
\end{align*}
as the standard Hellinger divergence between them. For a function $f:\bbR\to \bbR$ and a real number $\mu$, we denote the $\mu$-translated version of $f$ by $f_\mu:x\mapsto f(x-\mu)$. For two integers $m\leq n$, we use $[m:n]$ to denote the set $\{m,m+1,\cdots,n\}$; in particular, for $m=1$, we abbreviate it and write $[n]= \{1,2,\cdots,n\}$. For a cumulative distribution function $F$, we use the notation $F(A)$ where $A\subseteq \bbR$ to denote the probability of $X\in A$ when $X\sim F$.
 
\subsection{Related Work}\label{sec:related_work}

\paragraph{Location estimation adapted to unknown density shapes.}
Estimating the center of an unknown symmetric density is a classical semiparametric problem, with the location parameter as the target and the density shape as an infinite-dimensional nuisance. Beginning with \cite{stein1956efficient}, substantial literature, including \cite{van1970efficiency,stone1975adaptive,beran1978efficient, bickel1982adaptive}, has established that adapting to the unknown density need not incur any loss of asymptotic efficiency under suitable regularity conditions. See \cite{bickel1993efficient} for a systematic introduction. Subsequent work also characterizes the higher-order asymptotic terms \cite{mammen1997optimal,dalalyan2006penalized} under this semiparametric setup of adaptive location estimation. As can already be derived from these previous results and also the more recent paper \cite{laha2021adaptive}, there exists an estimator $\widehat{\mu}$ that is asymptotically minimax optimal, i.e., $\sqrt{n}(\widehat{\mu}-\mu)\overset{\mathrm{d}}{\to}\cN(0,I_f^{-1})$, for every $f\in\cF_\mathsf{SLC}$ with finite Fisher information. Compared to the aforementioned papers that target sharp asymptotic efficiency, we seek a constant-factor finite-sample guarantee at a prescribed confidence level, uniformly over a pre-specified shape-constrained class. Moreover, our analysis covers cases with infinite Fisher information, while the common asymptotic assumption of  finite Fisher information rules out the discussion of many compactly supported $f\in\cF_\mathsf{SLC}$, which are also interesting and meaningful instances among symmetric log-concave densities. 

More closely related to our formulation is a line of preceding literature that also investigates the finite-sample setup \cite{gupta2023finite, kao2024choosing, compton2026attainability}. The work \cite{compton2026attainability} has made significant progress towards adapting to symmetric log-concave tails, which is closest to the goal of the current paper. As the authors reveal, it is possible to achieve a rate of 
\begin{align*}
    r=\omega_f\left(\frac{C\log(n/\delta)\log^2(n)}{n}\right)
\end{align*}
for unknown and arbitrary $f\in\cF_\mathsf{SLC}$, which almost attains the instance-wise lower bound, except for extra $\operatorname{polylog}(n,1/\delta)$ terms. They also derive adaptation results for unknown mixtures of common-centered log-concave densities and show that it is impossible to adapt to all symmetric unimodal densities without the log-concavity restriction. Their methodology is to scan over candidate locations $\theta\in \bbR$ and set the estimator $\widehat{\mu}$ as any candidate $\theta$ around which the observed samples are almost left-versus-right balanced in a square-root-normalized sense. This procedure runs in $O(n\log(n)\log\log(n))$ time on sorted samples. The key observation is that the scale of square-root imbalance at a wrong center $\theta\neq \mu$ is, up to logarithmic factors, equal to $\HL^2(f_0, f_{2\abs{\theta-\mu}})$, which is derived from a reverse data-processing inequality in \cite{pensia2023communication}. We will discuss later in Section~\ref{sec:compare_rdp} that their logarithmic factors, inherited from the reverse data-processing inequality, are unavoidable for certain symmetric log-concave densities $f$ even at the population level.

The work \cite{kao2024choosing} moves from log-concavity to moment constraints. The authors propose an estimator named constraint asymptotic variance selection (CAVS) that minimizes the empirical $L^\gamma$ loss:
\begin{align*}
    \widehat{\mu}\in \argmin_{\theta\in\bbR}\frac{1}{n}\sum_{i=1}^n\abs{X_i-\theta}^\gamma,
\end{align*}
where $\gamma$ is chosen in a sample-dependent manner. The motivation is that different choices of $\gamma$ can recover the sample median ($\gamma=1$), the sample mean ($\gamma=2$), and the sample mid-range ($\gamma\to+\infty$); thus, a data-dependent $\gamma$ can potentially adapt to unknown shapes of densities. As the authors highlight, their procedure is nearly optimal, except for polylogarithmic factors, for a wide range of compactly supported symmetric location families that satisfy certain moment constraints. But it remains unknown whether such near-optimality holds universally for every symmetric log-concave shape. Another earlier paper \cite{gupta2023finite} achieves a rate related to Gaussian-smoothed Fisher information for arbitrary density, whereas it remains unclear whether and how this quantity is connected to the two-point testing lower bound. 

\paragraph{Location estimation with sample mid-summaries.}
Sample mid-summaries and their linear combinations are symmetric cases of L-estimators. Also known as quasi-mid-ranges or trimmed mid-ranges, sample mid-summaries have been applied to location estimation since a substantial body of early literature \cite{benson1949note, lloyd1952least, sen1961some, bickel1965some, crow1967robust, gastwirth1970small}. While these early works select a single mid-summary or a fixed combination of them for a specific set of known densities, a subsequent line of work \cite{jaeckel1971some, sacks1975asymptotically} proposes integrating the mid-summaries using data-dependent weights for unknown underlying densities. The adaptive mid-summary-based procedure therein can achieve optimal asymptotic variance for regular location families with finite Fisher information, whereas finite-sample optimality and the strategy for more general densities have not yet been established.

\paragraph{Log-concave and monotone-hazard-rate densities.}Log-concavity and monotone hazard rates (MHR) are classical shape restrictions with broad applications in reliability theory and economics, providing useful probabilistic structure without specifying a parametric distribution. Beyond the specific task of location estimation, prior work has addressed a wide range of topics related to log-concave distributions \cite{bagnoli2005log, balabdaoui2009limit, cule2010maximum, cule2010maximum, doss2016global, kim2016global, kim2018adaptation, kur2019optimality, xu2021high, feng2021adaptation, feng2026optimal}. Readers are referred to \cite{samworth2018recent, samworth2026nonparametric} for a comprehensive overview of recent progress. \cite{doss2016global}, \cite{kim2016global}, and \cite{kur2019optimality} characterize the minimax density estimation rate of log-concave densities. Specifically, in the one-dimensional setup, \cite{doss2016global} proves that the minimax density estimation rate is $\Theta(n^{-4/5})$ under the squared Hellinger metric. Therefore, recovering the shape of a log-concave density is genuinely nonparametrically hard, and the density-estimation rate can only imply a suboptimal $\omega_f(\Theta(n^{-4/5}))$-rate for location estimation. The work \cite{feng2026optimal} studies general M-estimation, including location estimation, under log-concave marginals; their focus is on asymptotic efficiency under regularity assumptions, whereas our formulation targets non-asymptotic rates and tolerates infinite Fisher information. The study of MHR distributions dates back to \cite{barlow1963properties, marshall1965maximum, rao1970estimation}. More recent studies \cite{dumbgen2017bi, laha2021bi} have shown that MHR is equivalent to bi-log-concavity, which is a weaker and more general shape constraint than the classical notion of log-concavity.

\section{Quantile-Based Characterization of Le Cam's Two-Point Rate}\label{sec:quantile_to_hellinger}
As we have illustrated earlier in Section~\ref{sec:intro}, when the noise profile $f$ is a known symmetric log-concave density, the minimax-optimal estimation error $r=r(n,\delta;f)$ under a failure probability $\delta\in(0,1/3]$ with $n\gtrsim \log(1/\delta)$ samples is entirely determined by the Hellinger-based critical equation:
\begin{align}\label{eq:critical_equation}
    \HL^2(f_0,f_{2r})\asymp \frac{\log(1/\delta)}{n}\quad \Longleftrightarrow \quad r= \omega_f\left(\Theta\Big(\frac{\log(1/\delta)}{n}\Big)\right),
\end{align}
where $\omega_f(\cdot)$ denotes the inverse Hellinger modulus of continuity defined by \eqref{eq:inverse_h2_modulus_of_continuity}. The above rate is matched from below via Le Cam's two-point testing lower bound and achieved from above by parametric MLE. In the unknown-$f$ setting, however, we must agnostically adapt to this $f$-dependent scaling in a sample-dependent manner, which significantly complicates the estimation task.

In this section, we present a surprising result: the minimax-optimal rate \eqref{eq:critical_equation} can be represented using only the quantiles of $f$. This quantile-based representation holds universally for all symmetric log-concave densities and motivates us to achieve instance-optimality by leveraging the distribution-free power of sample quantiles (see Section~\ref{sec:upper_bound}). More precisely, we show that the two-point rate \eqref{eq:critical_equation} can be expressed in terms of the multiscale function \eqref{eq:delta_intro}, which is based solely on the gaps between dyadic quantiles. 

\subsection{Quantifying Hellinger Continuity with A Multiscale Function of Quantiles}
In what follows, we formalize the function $\Delta_f(\cdot)$ mentioned in \eqref{eq:delta_intro} and prove its equivalence to the inverse Hellinger modulus of continuity $\omega_f(\cdot)$. For a starting percentile $p\in (0,1/4]$, we consider a dyadic list of percentiles:
\begin{align*}
    2^jp,\quad j=0,1,2,\cdots,J_p,
\end{align*}
where
\begin{align}\label{eq:jm}
    J_p\define\max\left\{j\in\bbN:\ 2^j p\leq 1/4\right\}
\end{align}
is the maximal index such that $2^jp$ is at most $1/4$. Let $F$ be the cumulative distribution function of $f\in\cF_\mathsf{SLC}$ and define
\begin{align*}
    q(u)=\inf\{x\in\bbR:\ F(x)\geq u\},\quad u\in(0,1)
\end{align*}
as its quantile function. We now define the dyadic quantile gap of $f$ by
\begin{align}\label{eq:quantile_gap}
    \ell(u)\define q(2u)-q(u),\quad u\in (0,1/2).
\end{align} 
Finally, we introduce a crucial quantity
\begin{align}\label{eq:delta_f_p}
    \Delta_f(p)=\left(\sum_{j=0}^{J_p}\frac{2^j}{\ell^2(2^j p)}\right)^{-1/2}.
\end{align}
Then the following result holds:
\begin{Proposition}\label{prop:critical_hellinger_by_quantile}
For any $p\in (0,1/4]$ and any $f\in\cF_\mathsf{SLC}$, we have
\begin{align}\label{eq:critical_hellinger_by_quantile}
    (1-1/e)\left(1\wedge \frac{(2t^2\wedge t)p}{1152}\right)\leq\HL^2\left(f_0,f_{2t\cdot\Delta_f(p)}\right)\leq  1\wedge (40t^2+4t)p
\end{align}
for all $t\geq 0$, where the quantity $\Delta_f(p)$ is defined as in \eqref{eq:delta_f_p}.
\end{Proposition}

Since $r\mapsto \HL^2(f_0,f_{2r})$ is non-decreasing (see Lemma~\ref{lem:su_preliminary}), Proposition~\ref{prop:critical_hellinger_by_quantile} indicates that
\begin{align*}
    \omega_f(cp)\leq \Delta_f(p)\leq \omega_f(Cp)
\end{align*}
for some universal $c,C>0$ and all $p$ below a constant threshold. Therefore, this quantile-based function $\Delta_f(\cdot)$ is a tight surrogate for the inverse Hellinger modulus of continuity $\omega_f(\cdot)$. In particular, setting $p\asymp \log(1/\delta)/n$ proves that $\Delta_f(\Theta(\log(1/\delta)/n))$ is equivalent to the instance-optimal rate \eqref{eq:critical_equation_intro} in the known-$f$ setting; hence, the adaptive instance-optimality question will be answered affirmatively if we can design an estimator that achieves the $\Delta_f(\Theta(\log(1/\delta)/n))$-rate without knowing the noise density $f$.

The detailed proof of Proposition~\ref{prop:critical_hellinger_by_quantile} is deferred to Appendix~\ref{sec:proof_quantile_to_hellinger}. In the remainder of this section, we present a high-level intuition for why Proposition~\ref{prop:critical_hellinger_by_quantile} is true. We will focus on explaining the upper bound (the second inequality of \eqref{eq:critical_hellinger_by_quantile}) when $t\asymp 1$, which directly implies the optimality of our estimator. Despite the formidable form of $\Delta_f(p)$ in \eqref{eq:delta_f_p}, we note that Proposition~\ref{prop:critical_hellinger_by_quantile} stems from the straightforward idea of approximating the Hellinger integral by replacing the integrand with its piecewise-constant surrogate. Piecewise constant approximation is not always tight for evaluating an arbitrary integral; nevertheless, it turns out to be tight in our evaluation of the Hellinger divergence due to the following property of a symmetric log-concave density:
\begin{Claim}[Flatness between quantiles]\label{clm:flatness_between_quantiles} Let $f\in \cF_\mathsf{SLC}$ and $\ell(\cdot)$ be the dyadic quantile gap function defined in \eqref{eq:quantile_gap}. For any $u\in(0,1/4]$, we have
\begin{align*}
    \min_{x\in [q(u),q(2u)]}\, f(x)\geq \frac{u}{2\ell(u)},\quad \max_{x\in [q(u),q(2u)]}\, f(x)\leq \frac{2u}{\ell(u)}
\end{align*}
\end{Claim}
The above result is contained within the proof of Lemma~\ref{lem:local_h2_upper_bound} in Appendix~\ref{sec:proof_quantile_to_hellinger}. As it implies, a symmetric log-concave density $f$ is almost \emph{flat} between the $u$-th and $2u$-th quantiles for every $u\in(0,1/4]$. To interpret this flatness claim, we first notice that the average of $f(x)$ over the interval $[q(u),q(2u)]$ is given by
\begin{align*}
    \frac{1}{q(2u)-q(u)}\int_{q(u)}^{q(2u)}f(x)\dif x=\frac{u}{\ell(u)}.
\end{align*}
This is where the dyadic quantile gap function $\ell(\cdot)$ emerges. Comparing this average to the bounds in Claim~\ref{clm:flatness_between_quantiles}, we find that $f(x)$ is almost flat on the interval $[q(u),q(2u)]$ in the sense that the interval minimum and the interval maximum are both equal to the interval average $u/\ell(u)$ up to a factor of $2$. On the same interval, we have 
\begin{align}\label{eq:cdf_flatness}
    F(x)\in [u,2u],\quad x\in [q(u),q(2u)]
\end{align}
by the definition of the quantile function and the cumulative distribution function. Specifically, we have $f(x)\asymp 2^jp/\ell(2^jp)$ and $F(x)\asymp 2^jp$ on each interval $[q(2^jp),q(2^{j+1}p)]$. The Hellinger integral can then be upper- and lower-bounded once we partition the real line into intervals $[q(2^jp),q(2^{j+1}p)]$ and control the integrand with $f$ and $F$ on each of these intervals.

Specifically, for a small shift $r>0$, we first divide the Hellinger integral into three parts:
\begin{align*}
    \HL^2\left(f_0,f_{2r}\right)&=\int_{q(p)+2r}^{-q(p)}\left(\sqrt{f(x)}-\sqrt{f(x-2r)}\right)^2\dif x\\
    &\ +\int_{-\infty}^{q(p)+2r}\left(\sqrt{f(x)}-\sqrt{f(x-2r)}\right)^2\dif x+\int_{-q(p)}^{+\infty}\left(\sqrt{f(x)}-\sqrt{f(x-2r)}\right)^2\dif x\\
    &\define S_\mathsf{central}+S_\mathsf{left}+S_\mathsf{right}.
\end{align*}
By the fundamental theorem of calculus and the Cauchy-Schwarz inequality,
\begin{align*}
    \left(\sqrt{f(x)}-\sqrt{f(x-2r)}\right)^2&=\left(\int_{x-2\Delta}^x 1\cdot (\sqrt{f})'(y)\dif y\right)^2\\
    &\leq 2r \int_{x-2\Delta}^x \left((\sqrt{f})'(y)\right)^2\dif y\\
    &=2r\cdot \int_{x-2\Delta}^{x}s^2(y)f(y)\dif y,
\end{align*}
where $s(x)=(\log f)'(x)$ is the score function defined almost everywhere within the support of $f$. Integrating both sides with respect to $x$ and exchanging the order of the integrals, we obtain
\begin{align*}
    S_\mathsf{central}\lesssim r^2 \int_{q(p)}^{-q(p)} s^2(x)f(x)\dif x=2r^2 \int_{q(p)}^{0} s^2(x)f(x)\dif x
\end{align*}
A crucial corollary of log-concavity is that $s(x)\leq f(x)/F(x)$ almost everywhere in the interior of the support (see Lemma~\ref{lem:monotone_reverse_hazard}). Applying this bound and telescoping the integral based on our dyadic partition $p_{m,j}=2^jp_{m,0}$, we arrive at
\begin{align*}
     S_\mathsf{central}\lesssim r^2 \int_{q(p)}^{0} \frac{f^3(x)}{F^2(x)}\dif x=2\sum_{j=0}^{J_p-1} \int_{q(2^jp)}^{q(2^{j+1}p)}\frac{f^3(x)}{F(x)}\dif x+2\int_{q(2^{J_p}p)}^{0}\frac{f^3(x)}{F(x)}\dif x.
\end{align*}
The rationale behind choosing the interval partition $\{[q(2^jp),q(2^{j+1}p)]\}_{j=0}^{J_p-1}$ is that both $F$ and $f$ are proportional to a known quantity on each of these intervals, which makes each integral controllable. Specifically, on the interval $[q(2^jp),q(2^{j+1}p)]$, we have $F(x)\asymp 2^jp$ by \eqref{eq:cdf_flatness} and also $f(x)\asymp 2^jp/\ell(p)$ by Claim~\ref{clm:flatness_between_quantiles}. Thus, $f^3(x)/F^2(x)\asymp 2^jp/\ell^3(2^jp)$ on the interval $[q(2^jp),q(2^{j+1}p)]$. The final interval $[q(2^{J_p}p),0]$ can be treated in a similar way. In total, we have
\begin{align*}
    S_\mathsf{central}\lesssim r^2 \sum_{j=0}^{J_p}\frac{(2^jp/\ell(2^jp))^3}{(2^jp)^2}\cdot \ell(2^jp)=r^2p\sum_{j=0}^{J_p}\frac{2^j}{\ell^2(2^jp)}= \frac{r^2p}{ \Delta_f^2(p)}.
\end{align*}
For the left tail $S_\mathsf{left}$, since $f(x)\geq f(x-2\Delta)$ in the range of integration to claim, we use $(\sqrt{b}-\sqrt{a})^2\leq b-a$ whenever $0\leq a\leq b$ to claim:
\begin{align*}
    S_\mathsf{left}\leq \int_{-\infty}^{q(p)+2r}\left( f(x)-f(x-2r)\right)\dif x&=F(q(p)+2r)-F(q(p))\\
    &\leq 2r f(q(p)+2r)\\
    &\leq  2r f(q(2p))
\end{align*}
for small $r$. Again, by Claim~\ref{clm:flatness_between_quantiles} that $f(q(p))\asymp p/\ell(p)$, we have
\begin{align*}
    f(q(p))\asymp \frac{p}{\ell(p)}=p\cdot \sqrt{\frac{1}{\ell^2(p)}}\leq p\cdot \sqrt{\sum_{j=0}^{J_p}\frac{2^j}{\ell^2(2^jp)}}=\frac{p}{\Delta_f(p)}.
\end{align*}
Therefore, it follows that $S_\mathsf{left}\lesssim rp/\Delta_f(p)$, which also holds for $S_\mathsf{right}$ by symmetry. Combining the three pieces of the integral, we conclude that
\begin{align*}
    \HL^2\left(f_0,f_{2r}\right)\lesssim p\left(\frac{r^2}{\Delta_f^2(p) }+  \frac{r}{\Delta_f(p)}\right).
\end{align*}
for all sufficiently small $r>0$. Setting $r=t\Delta_f(p)$, we obtain
\begin{align*}
     \HL^2\left(f_0,f_{2r}\right)\lesssim (t\vee t^2)p
\end{align*}
for every $t$ smaller than a universal threshold. The large-$t$ result can be attained from the triangle inequality
\begin{align*}
    \HL\left(f_0,f_{2t\Delta_f(m)}\right)\leq \sum_{i=0}^{N-1}\HL\left(f_{2t_i\Delta_f(m)},f_{2t_{i+1}\Delta_f(m)}\right),
\end{align*}
where $N= 1\vee\Theta(t)\in\bbZ_+$ is chosen appropriately and $0=t_0<t_1<\cdots<t_N=t$ is a partition of $[0,t]$ such that $\max_i\abs{t_i-t_{i+1}}\asymp 1$. This completes the proof of the upper bound in Proposition~\ref{prop:critical_hellinger_by_quantile}. The lower bound on Hellinger divergence can be tackled similarly using the quantile geometry of symmetric log-concave densities. The same upper bound on Hellinger divergence continues to hold for every symmetric unimodal density with a monotone hazard rate, while the Hellinger lower bound therein is modified to a weaker form.

\subsection{Comparison with Reverse Data-Processing Inequality}\label{sec:compare_rdp}
The near-instance-optimal estimator in the preceding work \cite{compton2026attainability} is defined as any point $\widehat{\mu}$ around which the observed samples are nearly balanced in a square-root normalized sense:
\begin{align*}
    \sup_{0\leq a\leq b\leq +\infty}\,\abs*{\sqrt{P_n([\widehat{\mu}-b,\widehat{\mu}-a])}-\sqrt{P_n([\widehat{\mu}+a,\widehat{\mu}+b])}}\lesssim \sqrt{\frac{\log(n/\delta)}{n}},
\end{align*}
where $P_n=(1/n)\sum_{i=1}^n\delta_{X_i}$ denotes the empirical distribution. This estimation procedure leverages a reverse data-processing inequality from \cite{pensia2023communication}, which claims the existence of a binary likelihood-ratio-based channel that preserves the Hellinger divergence up to logarithmic terms. For a log-concave density and its location-shifted version, the likelihood ratio between them is monotone in the interior of their common support and is $0$ or infinite otherwise. Therefore, for $f\in\cF_\mathsf{SLC}$, the reverse data-processing inequality in \cite{pensia2023communication} reduces to
\begin{align}\label{eq:data-processing_log}
    \sup_{-\infty\leq a\leq b\leq +\infty}\, \HL^2\left(\cT_{[a,b]}\circ f_0,\,\cT_{[a,b]}\circ f_{2r}\right)\gtrsim\frac{\HL^2\left(f_0,f_{2r}\right)}{\log\left(e/\HL^2\left(f_0,f_{2r}\right)\right)},
\end{align}
in which $\cT_{[a,b]}=\indi\{x\in[a,b]\}$ is the interval-thresholding operation, and $\cT_{[a,b]}\circ f$ denotes the distribution of the compound random variable $\cT_{[a,b]}(X)$ for $X\sim f$. The inequality \eqref{eq:data-processing_log} is one cause of the extra polylogarithmic factor in \cite{compton2026attainability}; the other cause is their normalized uniform convergence result regarding interval frequencies \cite{bousquet2003introduction,vapnik2015uniform}. We claim that the logarithmic term in inequality \eqref{eq:data-processing_log} is unavoidable among symmetric log-concave densities:
\begin{Claim}\label{clm:h2_log}
    The logarithmic factor in \eqref{eq:data-processing_log} is unavoidable for certain $f\in\cF_\mathsf{SLC}$. For example, with $f(x)=(1-\abs{x})_+$ being the triangle density, we have for $0\leq r\leq 1/2$ that
    \begin{align*}
        \HL^2\left(f_0,f_{2r}\right)\geq r^2\log(e/r),\quad \sup_{-\infty\leq a\leq b\leq +\infty}\,\HL^2\left(\cT_{[a,b]}\circ f_0,\, \cT_{[a,b]}\circ f_{2r}\right) \leq 2r^2.
    \end{align*}
\end{Claim}
The proof of the above claim is deferred to Appendix~\ref{sec:proof_h2_log}. As Claim~\ref{clm:h2_log} implies, the interval-thresholding operation, which reduces a real random variable to binary, may alter the local Hellinger geometry to a diverging extent, resulting in undesirable polylogarithmic factors in the sample complexity of location estimation. Our estimation procedure avoids this reduction approach by leveraging the log-free Hellinger-quantile inequalities \eqref{eq:critical_hellinger_by_quantile}, the tightness of which is ascribed to the multi-resolution (as opposed to binary) nature of the function $\Delta_f(\cdot)$. Furthermore, as we will illustrate in Section~\ref{sec:mhr}, the likelihood-ratio monotonicity of a log-concave location family is not a necessary condition for adapting to symmetric unimodal location families; the weaker condition of monotone hazard rates already suffices.

\section{Estimation by Adaptively Weighting the Sample Mid-Summaries}\label{sec:upper_bound}
Having established that $\Delta_f(\Theta(\log(1/\delta)/n))$ is equivalent to the known-$f$ optimal rate, we proceed to design an estimator that adaptively achieves this benchmark. In the current section, we introduce our estimation procedure, which is based on a convex combination of the sample mid-summaries at dyadic indices. The contents of this section are organized as follows:
\begin{enumerate}
    \item We start with the high-level intuition for determining the weights in the convex combination of sample mid-summaries. The weights are determined such that the convex combination is unbiased for $\mu$ and has minimal variance, where the variance-minimization problem closely depends on the shape of the noise density $f$.
    \item We then propose an estimator and sketch the proof that its high-probability fluctuation matches the $O(\Delta_f(\log(1/\delta)/n))$-rate. Key proof steps include representing the estimator as a function of independent exponential variables, a Bernstein-type bound for coordinate-wise Lipschitz functions aided by a truncation argument, and standard concentration results showing that the aforementioned optimal weights can be determined in a data-dependent manner.
\end{enumerate}

\subsection{A Variance-Minimizing Combination of the Sample Mid-Summaries}\label{sec:variance_minimization}
We now consider a dataset of size $m$ instead of $n$ since we will later perform sample-splitting and set $m=\floor{n/2}$. For $k,m\in \bbZ_+$ with $m\geq 4k$, we set
\begin{align*}
    p=\frac{k}{m+1}\in (0,1/4],\quad J_p=\min\{j\in\bbN:\ 2^jp\leq 1/4\},
\end{align*}
and define the dyadic integers $\{2^jk\}$ for $j=0,1,2,\cdots,J_p$.
Given a batch of $m$ i.i.d. samples from $f_\mu$ denoted by $X_{1:m}$, we consider their $2^jk$-th sample mid-summary:
\begin{align*}
    M_{2^jk}(X_{1:m})= \frac{X_{(2^jk)}+X_{(m+1-2^jk)}}{2},\quad j=0,1,2,\cdots,J_p.
\end{align*}
Our main idea is to recover the location $\mu$ with a convex combination of these sample mid-summaries, i.e., estimating with:
\begin{align*}
    \sum_{j=0}^{J_p}\lambda_j M_{2^jk},
\end{align*}
where $\lambda_{0:J_p}$ are convex weights that satisfy $\lambda_j\geq 0$ and $\sum_{j=0}^{J_p}\lambda_j=1$. While the weights $\lambda_{0:J_p}$ will eventually be computed in a data-dependent manner, the following paragraphs will illustrate how an ideal choice of the weights is determined from the shape of the underlying noise density $f$ at the population level.

By the symmetry of $f$, the unbiasedness of each mid-summary clearly holds:
\begin{align*}
    \E[M_{2^jk}]=\mu,\quad j=0,1,2,\cdots,J_p.
\end{align*}
Thus, their convex combination $\sum_{j=0}^{J_p}\lambda_j M_{2^jk}$ is also unbiased for $\mu$. Therefore, a straightforward criterion for choosing $\lambda_{0:J_p}$ is to minimize the variance of $\sum_{j=0}^{J_p}\lambda_j M_{2^jk}$. To evaluate the variance of this convex combination, we first look at each mid-summary $M_{2^jk}$. Let $F(\cdot)$ and $q(\cdot)$ be the cumulative distribution function and the quantile function of the density $f$, respectively. It is worth noting that, although commonly interpreted as the $2^jk/m$-th sample quantile, $X_{(2^jk)}$ is also the $2^jp=2^jk/(m+1)$-th sample quantile because $2^jk/m\geq 2^jp$ and $(2^jk-1)/m=2^jp-(m+1-2^jk)/(m(m+1))<2^jp$. By the Bahadur representation \cite{bahadur1966note, kiefer1967bahadur}, we have
\begin{align*}
    X_{(2^jk)}-\mu\approx q(2^jp)+\frac{2^jp-\frac{1}{m}\sum_{i=1}^m\indi\{X_i-\mu\leq q(2^jp)\}}{f(q(2^jp))}
\end{align*}
and similarly
\begin{align*}
    X_{(m+1-2^jk)}-\mu\approx q(1-2^jp)+\frac{(1-2^jp)-\frac{1}{m}\sum_{i=1}^m\indi\{X_i-\mu\leq q(1-2^jp)\}}{f(q(1-2^jp))}.
\end{align*}
The symmetry of $f$ implies $q(u)=-q(1-u)$ for every $u\in(0,1)$, and thus we can write
\begin{align*}
    M_{2^jk}\approx\mu+\frac{1-\frac{1}{m}\sum_{i=1}^m\indi\{X_i-\mu\leq q(2^jp)\}-\frac{1}{m}\sum_{i=1}^m\indi\{X_i-\mu\leq q(1-2^jp)\}}{f(q(2^jp))}.
\end{align*}
Note that the random variables in the numerator are simple binomials. It can then be computed that
\begin{align*}
    \Cov\left(M_{2^jk},M_{2^lk}\right)\approx  \frac{1}{f(q(2^jp))f(q(2^lp))}\cdot \frac{2(2^jk\wedge 2^lk)}{m(m+1)},\quad j,l=0,1,2,\cdots,J_p.
\end{align*}
As a first observation, the correlation between $M_{2^jk}$ and $M_{2^lk}$ is approximately
\begin{align}\label{eq:corr}
    \operatorname{Corr}\left(M_{2^jk},M_{2^l k}\right)\approx \frac{2^jk\wedge 2^l k}{\sqrt{2^jk\cdot 2^l k}}=2^{-\abs{j-l}/2}.
\end{align}
As a result, the dyadic mid-summaries $\{M_{2^jk}\}_{j=0}^{J_p}$ are nearly uncorrelated since the Toeplitz correlation matrix $(2^{-\abs{j-l}/2})_{j,l=0}^{J_p}$ is spectrally sandwiched between constant multiples of the identity matrix:
\begin{align}\label{eq:toeplitz_spectral}
    (3-2\sqrt{2})I_{J_p+1}\preceq \left(2^{-\abs{j-l}/2}\right)_{j,l=0}^{J_p}\preceq (3+2\sqrt{2})I_{J_p+1}.
\end{align}
Therefore, the variance of the convex combination $\sum_{j=0}^{J_p}\lambda_j M_{2^jk}$ can be controlled by
\begin{align}\label{eq:variance_approx}
    \Var\left[\sum_{j=0}^{J_p}\lambda_j M_{2^jk}\right]\approx \Theta(1)\cdot \sum_{j=0}^{J_p}\lambda_j^2\frac{2^jk}{m^2f^2(q(2^jp))}\asymp   \sum_{j=0}^{J_p}\lambda_j^2\frac{\ell^2(2^jp)}{2^jk},
\end{align}
where the last identity follows from the flatness result $f(q(u))\asymp u/\ell(u)$ discussed previously in Claim~\ref{clm:flatness_between_quantiles} of Section~\ref{sec:quantile_to_hellinger}.
With the convex constraint that $\lambda_j\geq 0$ and $\sum_{j=0}^{J_p}\lambda_j=1$, the quadratic form on the right-hand side is minimized by
\begin{align*}
    \lambda_j\propto \frac{2^jk}{\ell^2(2^jp)},\  \sum_{j=0}^{J_p}\lambda_j=1\quad \Longrightarrow\quad \lambda_j= \Delta_f^2(p)\frac{2^j }{\ell^2(2^jp)},
\end{align*}
with the minimized variance being approximately
\begin{align*}
   \Theta(1)\cdot \sum_{j=0}^{J_p}\left(\Delta_f^2(p)\frac{2^j}{\ell^2(2^jp)}\right)^2\frac{\ell^2(2^jp)}{2^jk}\asymp \Delta_f^4(p)\sum_{j=0}^{J_p}\frac{2^j}{\ell^2(2^jp)k}=\frac{\Delta_f^2(p)}{k}.
\end{align*}
Consequently, with this optimal choice of $\{\lambda_j\}_{j=0}^{J_p}$, setting $k\asymp 1$ and $m\asymp n$, we achieve, by applying Chebyshev's inequality, that
\begin{align*}
    \abs*{\sum_{j=0}^{J_p}\lambda_j M_{2^jk}-\mu}\approx O(1)\cdot \Delta_f(\Theta(1/n))
\end{align*}
with high probability. Though the preceding intuitions almost recover the desired optimal rate, several caveats are in order:
\begin{itemize}
    \item \textbf{Deviation bounds need rigorous justification.} The above derivation is only an intuitive explanation of the motivation behind our choice of the convex weights. The remainders in the Bahadur representation are not negligible; hence, the preceding derivation is not a rigorous finite-sample control of variances and covariances. Fortunately, these issues affect neither the choice of the optimal weights nor the claimed deviation bound, as we will formalize the nonasymptotic analysis in Section~\ref{sec:upper_bound_fixed_confidence}. It is further required to keep track of the dependence on $\delta$ so that the upper bound applies tightly to arbitrary failure probability $\delta\in(0,1/3]$.
    \item \textbf{Estimating weights from the holdout samples.} The optimal choice of the weights $\lambda_j\propto k_j/\ell^2(p_{m,j})$ depends on the unknown dyadic quantile gaps and is therefore not applicable to a real adaptive estimator. However, we may still use their plug-in versions, which are equivalent up to constants with high probability and only enlarge the deviation bound by a universal constant. To avoid the difficulty arising from the dependence between weights and sample mid-summaries, we use a standard sample-splitting technique in which we hold out the first half of the data for plugged-in weights and then extract the sample mid-summaries from the remaining samples. See Section~\ref{sec:upper_bound_fixed_confidence} for the use of this technique.
\end{itemize}

\subsection{Finite-Sample Analysis and The Design of Estimator}\label{sec:upper_bound_fixed_confidence}
Following the preceding discussion of motivations, we proceed to design an estimator and prove that it achieves a rate of $O(1)\cdot \Delta_f(\Theta(\log(1/\delta)/n))$ with $1-\delta$ confidence. Again, consider a batch of $m$ i.i.d. samples $\cD_m=\{X_i\}_{i=1}^m$ following from $f_\mu$ for some unknown $\mu\in\bbR$ and $f\in\cF_\mathsf{SLC}$. Our first target is a rigorous high-probability deviation bound for an arbitrary convex combination of the dyadic sample mid-summaries:
\begin{align}\label{eq:convex_combination_variance_bound}
    \sum_{j=0}^{J_p}\lambda_j M_{2^jk},
\end{align}
where $M_{2^jk}=(X_{(2^jk)}+X_{(m+1-2^jk)})/2$ is the $2^jk$-th sample mid-summary, $k\leq m/4$ is an $\delta$-dependent integer to be specified later, $p=k/(m+1)$, $J_p=\max\{j:\,2^jp\leq 1/4\}$, and $\lambda_{0:J_p}$ are a set of convex weights. The key tool is the following Bernstein-type concentration inequality that can be derived from the results in \cite{bobkov1997poincare}.
\begin{Lemma}\label{lem:efron_stein_coordinatewise_lipschitz}
Let $N$ be a positive integer and $E_{1:N}$ be independent $\mathrm{Exp}(1)$ random variables. Suppose a measurable function $H:\bbR_+^N\to \bbR$ satisfies the following coordinate-wise Lipschitz condition: for each $i\in [N]$ and any $v,v'\in\bbR_+^N$ that differ only in their $i$-th coordinate,
\begin{align*}
    \abs{H(v)-H(v')}\leq L_i \abs{v_i-v_i'}
\end{align*}
for some constant $L_i\in [0,+\infty)$. Then, for any $t\geq 0$,
\begin{align*}
    \Prob_{E_{1:N}\iid\mathrm{Exp}(1)}\left(\abs*{H(E_{1:N})-\E[H(E_{1:N})]}> 2\sqrt{t\sum_{i=1}^N L_i^2}+t\max_{i\in[N]}L_i\right)\leq 2e^{-t}.
\end{align*}
\end{Lemma}
To apply Lemma~\ref{lem:efron_stein_coordinatewise_lipschitz}, we need to rewrite the convex combination \eqref{eq:convex_combination_variance_bound} as a coordinate-wise Lipschitz function of independent $\mathrm{Exp}(1)$ random variables. The main steps are as follows: 
\paragraph{Step 1. Independent exponential representation.}The first step is to rewrite the estimator as a function of $(m+1)$ i.i.d. $\mathrm{Exp}(1)$ random variables to fulfill the independence requirement of Lemma~\ref{lem:efron_stein_coordinatewise_lipschitz}. Let 
\begin{align*}
    U_i=F(X_i-\mu)\iid \mathrm{Unif}(0,1),\quad i=1,2,\cdots,m.
\end{align*} 
Since $X_{1:m}$ are almost surely within the interior of the support of $f_\mu$ where $q(F(x))=x$ holds, we have $X_{(i)}=\mu+q(U_{(i)})$ for all $i\in[m]$ almost surely. Therefore, the convex combination \eqref{eq:convex_combination_variance_bound} can be written as a random function whose randomness depends only on $\{U_{(i)}\}_{i=1}^m$. Given the uniform order statistics $\{U_{(i)}\}_{i=1}^m$, with an independent auxiliary random variable $G\sim \Gamma(m+1,1)$, the vector
\begin{align}\label{eq:e_to_u}
    \left(E_1,E_2,\cdots,E_{m+1}\right)\define \left(GU_{(1)},G(U_{(2)}-U_{(1)}), G(U_{(3)}-U_{(2)}),\cdots,G(1-U_{(m)})\right)\in\bbR_+^{m+1}
\end{align}
has independent $\mathrm{Exp}(1)$ entries. Let
\begin{align*}
    S_i^-=\sum_{t=1}^i E_t,\quad S_i^+=\sum_{t=1}^i E_{m+2-t},
\end{align*}
be the $i$-th partial sums of $E_{1:(m+1)}$ from the left and from the right, respectively; set $S_{m+1}=\sum_{t=1}^{m+1} E_t$ as the full sum. Then, each uniform order statistic can be represented using $E_{1:(m+1)}$ by
\begin{align*}
    U_{(i)}=\frac{G\cdot\left(U_{(1)}+ \sum_{t=1}^i (U_{(t)}-U_{(t-1)})\right)}{G\cdot\left(U_{(1)}+ \sum_{t=1}^m (U_{(t)}-U_{(t-1)})+(1-U_{(m)})\right)}=\frac{\sum_{t=1}^{i} E_t}{\sum_{t=1}^{m+1}E_t}=\frac{S^-_i}{S_{m+1}}.
\end{align*}
Note how the above representation is free of the auxiliary variable $G$. Consequently, we can express each $U_{(i)}$ as a function of independent $\mathrm{Exp}(1)$ random variables $E_{1:(m+1)}$, and thus each $X_{(i)}$ and the convex combination \eqref{eq:convex_combination_variance_bound} can also be expressed with $E_{1:(m+1)}$. Specifically, we have
\begin{align}\label{eq:mid_by_exponential}
    M_{2^jk}=\mu+\frac{q\left(U_{(2^jk)}\right)+q\left(U_{(m+1-2^jk)}\right)}{2}=\mu+\frac{q\left(S^-_{2^jk}/S_{m+1}\right)-q\left(S^+_{2^jk}/S_{m+1}\right)}{2},
\end{align}
and thus
\begin{align}\label{eq:convex_combination_by_exponential}
    \sum_{j=0}^{J_p}\lambda_j M_{2^jk}=\mu+\sum_{j=0}^{J_p}\lambda_j\cdot \frac{q\left(S^-_{2^jk}/S_{m+1}\right)-q\left(S^+_{2^jk}/S_{m+1}\right)}{2}
\end{align}
for every set of convex weights $\lambda_{0:J_p}$.

\paragraph{Step 2. Coordinate-wise Lipschitzness via high-probability truncation.} As seen from \eqref{eq:mid_by_exponential}, the dyadic sample mid-summaries depend solely on the i.i.d exponential variables through the quantile function $q$. With log-concavity, we can claim the following local Lipschitzness of $q$, which is included in Lemma~\ref{lem:quantile_gap_stability} of Appendix~\ref{sec:proof_quantile_to_hellinger}.
\begin{Claim}\label{clm:quantile_lipschitz}
Let $f\in\cF_\mathsf{SLC}$ and let $q$ be its quantile function. Then, for any $u\in(0,1/4]$, the quantile function $q$ is $2/f(q(u))$-Lipschitz on the interval $[u/2,2u]$.
\end{Claim}
However, each dyadic mid-summary $M_{k_j}$ does not unconditionally satisfy the desired coordinate-wise Lipschitz condition required by Lemma~\ref{lem:efron_stein_coordinatewise_lipschitz} because the quantile function $q$ can diverge near $0^+$ and $1^-$, and its ratio arguments $S^\pm_{2^jk}/S_{m+1}$ can be close to $0$ or $1$ on some extreme events. However, since $\E[S^\pm_i]=i$, with a union of Chernoff bounds for Gamma random variables (see Lemma~\ref{lem:gamma_concentration}), we conclude that the concentration event
\begin{align}\label{eq:good_event}
    \cE(\cD_m)\define \left\{\abs*{\frac{S_{m+1}}{m+1}-1}\leq \frac{1}{10},\ \abs*{\frac{S^\pm_{2^jk}}{2^jk}-1}\leq \frac{1}{10},\ j=0,1,2,\cdots,J_p+1\right\}
\end{align}
holds with a probability of at least $1-c_1\exp(-c_2m)-c_3\sum_j\exp(-c_42^jk)$. This good event holds with probability $1-\delta/3$ as long as $k\geq A\log(1/\delta)$ for a sufficiently large universal constant $A>0$ and $m\geq 4k$. Importantly, on this high-probability event $\cE$, the random ratios $S^\pm_{2^jk}/S_{m+1}$ are equivalent to the ratio of the means $2^jk/(m+1)=2^jp$ within a factor of $2$, i.e., $S^\pm_{2^jk}/S_{m+1}\in[2^{j-1}p,2^{j+1}p]$ for all $j=0,1,\cdots,J_p$, which falls into the scope of Claim~\ref{clm:quantile_lipschitz}. To facilitate the analysis, we introduce an argument-truncated version of the residual $\sum_{j=0}^{J_p}\lambda_j M_{2^jk}-\mu$, defined by
\begin{align}\label{eq:truncated_residual}
    Y_\lambda\define  \sum_{j=0}^{J_p}\lambda_j\cdot \frac{q\left(\frac{0.9(2^jk)\vee S^-_{2^jk}\wedge 1.1(2^jk)}{0.9(m+1)\vee S_{m+1}\wedge 1.1(m+1)}\right)-q\left(\frac{0.9(2^jk)\vee S^+_{2^jk}\wedge 1.1(2^jk)}{0.9(m+1)\vee S_{m+1}\wedge 1.1(m+1)}\right)}{2}.
\end{align}
Despite the heavy notation in \eqref{eq:truncated_residual}, a comparison between \eqref{eq:truncated_residual} and \eqref{eq:convex_combination_by_exponential} shows that $Y_\lambda=\sum_{j=0}^{J_p}\lambda_j M_{2^jk}-\mu$ on the high-probability good event $\cE$ where all truncation operations in \eqref{eq:truncated_residual} are inactive. Therefore, to control the high-probability deviation of the convex combination $\sum_{j=0}^{J_p}\lambda_j M_{2^jk}$, it suffices to bound $Y_\lambda$, which, according to Lemma~\ref{lem:efron_stein_coordinatewise_lipschitz}, can be completed by verifying the coordinate-wise Lipschitz property of $Y_\lambda=Y_\lambda(E)$ with respect to $E_{1:(m+1)}$. For the non-truncated ratio $S^-_{2^jk}/S_{m+1}$, we have
\begin{align*}
    \abs*{\frac{\partial}{\partial E_i}\frac{S^-_{k_j}}{S_{m+1}}}=\begin{cases}
        \frac{S_{m+1}-S_{2^jk}}{S_{m+1}^2},\ &i\leq 2^jk;\\
        \frac{S^-_{2^jk}}{S_{m+1}^2},\ &i>2^jk.
    \end{cases}
\end{align*}
When the truncations are inactive, the range of $S_{m+1}$ and $S_{2^jk}$ implies that the ratio is $O(1/m)$-Lipschitz when $i\leq 2^jk$ and $O(2^jk/m^2)=O(2^jp/m)$-Lipschitz when $i>2^jk$. When one or both of the numerator and denominator truncations are active, the same Lipschitz condition holds, as we will show in detail in Appendix~\ref{sec:proof_truncated_lipschitz}.
The conclusion holds analogously for $S^+_{2^jk}/S_{m+1}$ by replacing $i$ with $m+2-i$. Meanwhile, note that the truncated inner argument lies within $[2^{j-1}p, 2^{j+1}p]$, on which $q$ is $2/f(q(2^jp))$-Lipschitz. Therefore, the total Lipschitzness of $Y_\lambda$ is controlled by:
\begin{Lemma}\label{lem:truncated_lipschitz}
    For any $i\in[m+1]$, if two vectors $E,E'\in\bbR_+^{m+1}$ only differ in their $i$-th coordinate, then
    \begin{align*}
        \abs{Y_\lambda(E)-Y_\lambda(E')}&\leq L_i\abs{E_i-E'_i},
    \end{align*}
    where
    \begin{align*}
        L_i&=\frac{2}{m+1}\left(\sum_{j: 2^jk\geq i}\frac{\lambda_j}{f(q(2^jp))}+\sum_{j: 2^jk\geq m+2-i}\frac{\lambda_j}{f(q(2^jp))}\right)\\
        &\ +\frac{2}{m+1}\left(\sum_{j: 2^jk< i}\frac{\lambda_j\cdot 2^jp}{f(q(2^jp))}+\sum_{j: 2^jk< m+2-i}\frac{\lambda_j\cdot 2^jp}{f(q(2^jp))}\right).
    \end{align*}
\end{Lemma}
\paragraph{Step 3. Aggregating the Lipschitz constants.}
Let $L_{1:(m+1)}$ be the Lipschitz constants as in Lemma~\ref{lem:truncated_lipschitz}. Using $f(q(u))\asymp u/\ell(u)$ from Claim~\ref{clm:flatness_between_quantiles}, we obtain
\begin{align*}
    \max_i L_i\leq \frac{4}{m+1}\sum_{j=0}^{J_p}\frac{\lambda_j}{f(q(2^jp))}\asymp \frac{1}{k}\sum_{j=0}^{J_p}\frac{\lambda_j\ell(2^jp)}{2^j}.
\end{align*}
For the quadratic sum, we have
\begin{align*}
    &\ \sum_{i=1}^{m+1}L_i^2\\
    &\lesssim \frac{1}{m^2}\sum_{i=1}^{m+1}\left(\sum_{j: 2^jk\geq i}\frac{\lambda_j}{f(q(2^jp))}+\sum_{j: 2^jk\geq m+2-i}\frac{\lambda_j}{f(q(2^jp))}+\sum_{j: 2^jk< i}\frac{\lambda_j2^jp}{f(q(2^jp))}+\sum_{j: 2^jk< m+2-i}\frac{\lambda_j2^jp}{f(q(2^jp))}\right)^2\\
    &\asymp \frac{1}{m^2}\sum_{i=1}^{m+1}\left(\sum_{j: 2^jk\geq i}\frac{\lambda_j}{f(q(2^jp))}\right)^2+\frac{1}{m^2}\sum_{i=1}^{m+1}\left(\sum_{j: 2^jk\geq m+2-i}\frac{\lambda_j}{f(q(2^jp))}\right)^2\\
    &\ +\frac{1}{m^2}\sum_{i=1}^{m+1}\left(\sum_{j: 2^jk< i}\frac{\lambda_j2^jp}{f(q(2^jp))}\right)^2+\frac{1}{m^2}\sum_{i=1}^{m+1}\left(\sum_{j: 2^jk< m+2-i}\frac{\lambda_j2^jp}{f(q(2^jp))}\right)^2\\
    &\define T_1+T_1'+T_2+T_2',
\end{align*}
where we have used $(a+b+c+d)^2\leq 4(a^2+b^2+c^2+d^2)$. The above terms are not yet in a tractable form. We simplify them by expanding the squares using the following algebraic inequality:
\begin{Lemma}\label{lem:dyadic_ineq} Let $N, k\in\bbZ_+$. Further, let $J\in\bbZ_+$ and $p_\star\in (0,1)$ be such that $2^Jk\leq p_\star N$. Then, for arbitrary real numbers $b_{0:J}$,
\begin{align*}
    \sum_{i=1}^N\left(\sum_{0\leq j\leq J:\,2^jk\geq i} b_j\right)^2\leq (3+2\sqrt{2})k\sum_{j=0}^J 2^j b_j^2;
\end{align*}
and for arbitrary nonnegative $b_{0:J}$,
\begin{align*}
    \sum_{i=1}^N\left(\sum_{0\leq j\leq J:\,2^jk<i}\frac{2^j b_j}{N}\right)^2\leq 2p_\star \sum_{j=0}^J 2^j b_j^2.
\end{align*}
\end{Lemma}
With their proof deferred to Appendix~\ref{sec:proof_upper_bound}, the above inequalities follow from the same fact as \eqref{eq:corr} and \eqref{eq:toeplitz_spectral}: a Toeplitz correlation structure is nearly uncorrelated. Take the first inequality as an example. We can write
\begin{align*}
     \sum_{i=1}^N\left(\sum_{0\leq j\leq J:\,2^jk\geq i} b_j\right)^2=\sum_{j,l}b_j b_l\sum_{i=1}^N \indi\{i\leq 2^jk\}\indi\{i\leq 2^lk\}&=k\sum_{j,l}(2^j\wedge 2^l)b_j b_l\\
     &=k\sum_{j,l}2^{-\abs{j-l}/2}\sqrt{2^jb_j^2}\sqrt{2^l b_l^2}.
\end{align*}
The first inequality then follows from the fact that $(2^{-\abs{j-l}/2})_{j,l=0}^{N}\preceq (3+2\sqrt{2})I_{N+1}$. We can now set $N=m+1$, $J=J_p$, $p_\star=1/4$, and $b_j=\lambda_j/f(q(2^jp))$ in Lemma~\ref{lem:dyadic_ineq} to obtain
\begin{align*}
    \max\{T_1,T_2\}\lesssim \frac{1}{m^2}\sum_{j=0}^{J_p}\frac{\lambda_j^2(2^jk)}{f^2(q(2^jp))}\asymp\sum_{j=0}^{J_p}\frac{\lambda_j^2\ell^2(2^jp)}{2^jk},
\end{align*}
where the last identity follows from Claim~\ref{clm:flatness_between_quantiles} that $f(q(u))\asymp u/\ell(u)$.
The same bounds hold for $T_1'$ and $T_2'$ by replacing $i$ with $\overline{i}=m+2-i$. We have therefore reached:
\begin{Lemma}\label{lem:truncated_mean_and_variance}
For any set of convex weights $\lambda_{0:J_p}$ and the $L_{1:(m+1)}$ defined in Lemma~\ref{lem:truncated_lipschitz},
\begin{align}\label{eq:max_and_square}
     \max_{i\in[m+1]}L_i\leq \frac{8}{k}\sum_{j=0}^{J_p}\frac{\lambda_j \ell(2^jp)}{2^j} ,\quad \sum_{i=1}^{m+1}L_i^2\leq \frac{1024}{k}\sum_{j=0}^{J_p}\frac{\lambda_j^2\ell^2(2^jp)}{2^j}.
\end{align}
\end{Lemma}
Note that the above rigorous variance bound coincides with our previous intuitive derivation \eqref{eq:variance_approx} in Section~\ref{sec:variance_minimization}. If the weights $\lambda_{0:J_p}$ are optimally set to
\begin{align*}
    c\Delta_f^2(p)\frac{2^j}{\ell^2(2^jp)}\leq \lambda_j\leq C\Delta_f^2(p)\frac{2^j}{\ell^2(2^jp)},\quad j=0,1,2,\cdots,J_p
\end{align*}
as in Section~\ref{sec:variance_minimization} for some universal $c,C>0$, then \eqref{eq:max_and_square} yields
\begin{align*}
    \max_{i\in[m+1]}L_i\lesssim \frac{1}{k}\sum_{j=0}^{J_p}\frac{\Delta_f^2(p)2^j}{\ell^2(2^jp)}\cdot\frac{\ell(2^jp)}{2^j}\leq \frac{\Delta_f^2(p)}{k}\sqrt{\sum_{j=0}^{J_p}\frac{2^j}{\ell^2(2^jp)}}\cdot \sqrt{\sum_{j=0}^{j_p}2^{-j}}\lesssim \frac{\Delta_f(p)}{k},
\end{align*}
and
\begin{align*}
    \sum_{i=1}^{m+1}L_i^2\lesssim \frac{1}{k}\sum_{j=0}^{J_p}\Delta_f^4(p)\frac{(2^j)^2}{\ell^4(2^jp)}\cdot\frac{\ell^2(2^jp)}{2^j}= \frac{\Delta_f^2(p)}{k}.
\end{align*}
Since $\E[Y_\lambda]=0$ by symmetry, according to Lemma~\ref{lem:efron_stein_coordinatewise_lipschitz}, we have
\begin{align*}
    \Prob\left(\abs{Y_\lambda}\leq C_1\sqrt{\frac{\Delta_f^2(p)\log(1/\delta)}{k}}+C_2\frac{\Delta_f(p)\log(1/\delta)}{k}\right)\geq 1-\delta.
\end{align*}
Set $k\asymp \log(1/\delta)$ and $m\asymp n$, and thus $p=k/(m+1)\asymp \log(1/\delta)/n$. Then, the right-hand side in the above bound becomes precisely the desired scaling of $\Delta_f(\Theta(\log(1/\delta)/n))$. However, such an optimal choice of $\lambda$ depends on the dyadic quantile gaps $\{\ell(2^jp)\}_{j=0}^{J_p}$ of the unknown density $f$; we have to determine them in a sample-dependent manner through an additional step below:
\paragraph{Step 4. Determining weights from hold-out data.}From now on, we will apply a sample-splitting strategy. For $n$ iid samples from $f_\mu$, we divide them equally into a training set $\cD_m=\{X_{1:m}\}$ and a hold-out set $\cD'_m=\{X'_{1:m}\}$, where $m=\floor{n/2}$. We also set $k=\ceil{720\log(1/\delta)}$. Recall that $p=k/(m+1)$ and
\begin{align*}
    \ell(2^jp)=q(2^{j+1}p)-q(2^jp),\quad j=0,1,2,\cdots,J_p
\end{align*}
is the population-level gap between the $2^jp$-th and the $2^{j+1}p$-th quantiles. We define its empirical plug-in estimate $\widehat{\ell}_{j}$ by
\begin{align*}
    \widehat{\ell}_{j}=\frac{\widehat{\ell}_{j}^-+\widehat{\ell}_{j}^+}{2},\quad \textnormal{where}\ \widehat{\ell}_{j}^-=X'_{(2^{j+1}k)}-X'_{(2^jk)},\ \widehat{\ell}_{j}^+=X'_{(m+1-2^jk)}-X'_{(m+1-2^{j+1}k)}.
\end{align*}
We then use $\{\widehat{\ell}_j\}_{j=0}^{J_p}$ to define the plug-in estimates for the unnormalized and normalized weights:
\begin{align*}
    \widehat{w}_j=\frac{2^j}{\widehat{\ell}_j^2},\quad \widehat{\lambda}_j=\frac{\widehat{w}_j}{\sum_{t=0}^{J_p}\widehat{w}_t},\quad j=0,1,2,\cdots,J_p.
\end{align*}
The following lemma, proved in Appendix~\ref{sec:proof_gap_and_weight_stable}, states that these plug-in estimates are comparable to their population target up to universal constants with high probability:
\begin{Lemma}\label{lem:gap_and_weight_stable}
    Let $k=\ceil{720\log(1/\delta)}$, $m\geq 4k$, and $p=k/(m+1)$. On an event $\cE'=\cE'(\cD'_m)$ that holds with probability $\Prob_{\cD'_m}(\cE')\geq 1-\delta/3$, we have simultaneously for all $j=0,1,2,\cdots,J_p$ that
    \begin{align*}
        \frac{1}{4}\ell(2^jp)\leq \widehat{\ell}_j\leq 4\ell(2^jp),\quad \frac{1}{256}\Delta_f^2(p)\frac{2^j}{\ell^2(2^jp)}\leq \widehat{\lambda}_j\leq 256\Delta_f^2(p)\frac{2^j}{\ell^2(2^jp)}.
    \end{align*}
\end{Lemma}
We remark that Lemma~\ref{lem:gap_and_weight_stable} follows from standard concentration results for Gamma random variables, and the event $\cE'$ is in the same form as $\cE$ in \eqref{eq:good_event} with $\cD_m$ replaced by the hold-out data $\cD'_m$. Finally, we define an estimator $\widehat{\mu}_\delta=\widehat{\mu}_\delta(\cD_m,\cD'_m)$ by
\begin{align}\label{eq:estimator}
    \widehat{\mu}_\delta=\sum_{j=0}^{J_p} \widehat{\lambda}_j(\cD'_m)\cdot  M_{2^jk}(\cD_m).
\end{align}
With this estimator, we claim that:
\begin{Proposition}\label{prop:estimator}
    For any $\delta\in(0,1/3]$ and the estimator $\widehat{\mu}_\delta$ defined by \eqref{eq:estimator}, we have
    \begin{align*}
        \inf_{\mu\in\bbR,\, f\in\cF_\mathsf{SLC}}\Prob_{X_{1:n}\iid f_\mu}\left(\abs{\widehat{\mu}_\delta-\mu}\leq 850\Delta_f\Big(\frac{\ceil{720\log(1/\delta)}}{\floor{n/2}+1}\Big)\right)\geq 1-\delta
    \end{align*}
    whenever $n\geq 8\ceil{720\log(1/\delta)}$.
\end{Proposition}
The constants in Proposition~\ref{prop:estimator} are not optimized. See Algorithm~\ref{alg:estimator} for a pseudo-code summary of the entire workflow, where our estimator~\eqref{eq:estimator} is realized by setting $k\asymp\log(1/\delta)$. Detailed proofs of Lemma~\ref{lem:truncated_lipschitz}, Lemma~\ref{lem:dyadic_ineq}, Lemma~\ref{lem:truncated_mean_and_variance},  Lemma~\ref{lem:gap_and_weight_stable}, and Proposition~\ref{prop:estimator} are all deferred to Appendix~\ref{sec:proof_upper_bound}. With Proposition~\ref{prop:estimator}, we are ready to prove the main theorem for symmetric log-concave cases:
\begin{proof}[Proof of Theorem~\ref{thm:main} (symmetric log-concave cases)]
Plugging the estimator's upper bound in \eqref{prop:estimator} into the Hellinger-quantile inequality \eqref{eq:critical_hellinger_by_quantile} in Proposition~\ref{prop:critical_hellinger_by_quantile}, we obtain
\begin{align*}
    \HL^2\left(f_0,f_{2\cdot 850 \Delta_f\left(\ceil{720\log(1/\delta)}/(\floor{n/2}+1)\right)}\right)\leq \frac{C\log(1/\delta)}{2n}<\frac{C\log(1/\delta)}{n}
\end{align*}
whenever $n\geq C'\log(1/\delta)$, where $C$ and $C'$ are universal absolute constants. Since $r\mapsto \HL^2(f_0,f_{2r})$ are non-decreasing (see Lemma~\ref{lem:su_preliminary}), we conclude that the rate in Proposition~\ref{prop:estimator} is upper bounded by $\omega_f(C\log(1/\delta)/n)$, which completes the proof.
\end{proof}
\begin{algorithm}[t]
\caption{Adaptive location estimation by aggregating dyadic sample mid-summaries}
\label{alg:estimator}
\begin{algorithmic}[1]
\Require Observations $X_1,\ldots,X_n\in\mathbb R$; an initial rank $k\in\bbZ_+$ where $ 1\leq k\leq \floor{n/4}$.
\Ensure A location estimate $\widehat\mu$ for $\mu$.

\State $m\gets\lfloor n/2\rfloor$, \quad
    $p\gets k/(m+1)$,\quad $J_p\gets\lfloor\log_2(1/(4p))\rfloor$.

\State Form $\cD_m=(X_1,\cdots,X_m)$ and
    $\cD'_m=(X'_1,\cdots,X'_m)$, where $X'_i=X_{m+i}$.
\State Sort $\cD_m$ and $\cD'_m$ separately; denote their order
    statistics by $X_{(i)}$ and $X'_{(i)}$, respectively.
    
\For{$j=0,\ldots,J_p$}
    \State $\widehat\ell_j^-\gets
        X'_{(2^{j+1}k)}-X'_{(2^jk)}$.
    \State $\widehat\ell_j^+\gets
        X'_{(m+1-2^jk)}-X'_{(m+1-2^{j+1}k)}$.
    \State $\widehat\ell_j\gets
        (\widehat\ell_j^-+\widehat\ell_j^+)/2$.
    \State $\widehat w_j\gets 2^j/\widehat\ell_j^{\,2}$.
    \State $M_{2^jk}\gets
        (X_{(2^jk)}+X_{(m+1-2^jk)})/2$.
\EndFor

\State $\widehat\lambda_j\gets\widehat w_j/\sum_{j=0}^{J_p}\widehat w_j$
    for $j=0,\ldots,J_p$.
\State $\widehat\mu\gets
    \sum_{j=0}^{J_p}\widehat\lambda_j M_{2^jk}$.
\State \Return $\widehat\mu$.
\end{algorithmic}
\end{algorithm}

\section{Numerical Studies}\label{sec:numerical_study}
This section aims to complement our theoretical analysis with numerical evidence. We experiment with $6$ symmetric log-concave densities with different scaling behaviors in their inverse Hellinger modulus of continuity:
\begin{enumerate}
    \item Gaussian distribution $f(x)=(1/\sqrt{2\pi})\exp(-x^2/2)$. This is the most well-known example of a symmetric log-concave density with finite Fisher information and root-$n$ rate. We have $\omega_f(t)\asymp \sqrt{t}$, and thus the optimal rate is $\Theta_\delta(n^{-1/2})$.
    \item Uniform distribution $f(x)=(1/2)\indi\{\abs{x}\leq 1\}$. This is the most extreme example of a density that exhibits an entirely boundary-driven rate faster than root-$n$. We have $\omega_f(t)\asymp t$, and thus the optimal rate is $\Theta_\delta(n^{-1})$.
    \item Semicircle distribution $f(x)=(2/\pi)\sqrt{(1-x^2)_+}$. This serves as an example of a compactly supported density with power boundaries. We have $\omega_f(t)\asymp t^{2/3}$, and thus the optimal rate is $\Theta_\delta(n^{-2/3})$.
    \item Triangle distribution $f(x)=(1-\abs{x})_+$. This density is linear near the support boundary, and its rate is almost root-$n$ except for an extra logarithmic term in the denominator. We have $\omega_f(t)\asymp \sqrt{t/\log(1/t)}$, and thus the optimal rate is $\Theta_\delta((n\log(n))^{-1/2})$.
    \item Epanechnikov $f(x)=(3/4)(1-x^2)_+$. This is a well-known density in the nonparametric literature since it is the optimal smoothing kernel that minimizes the asymptotic mean integrated squared error \cite{epanechnikov1969non}. Compared to triangle density, it is smooth in the interior but has the same scaling of the optimal rate.
    \item Power-log distribution
    \begin{align*}
        f_{\alpha,\kappa}(x)=\frac{1}{Z_{\alpha,\kappa}}\cdot \frac{(1-x^2)^{\alpha}}{\left[\log\left(\frac{e}{1-x^2}\right)\right]^{\kappa}}\cdot \indi\{\abs{x}< 1\},
    \end{align*}
    where $\alpha\in(0,1)$ and $\kappa>0$, $Z_{\alpha,\kappa}$ is a normalizing constant. For this density, $\omega_f(t)\asymp_{\alpha,\kappa}t^{1/(1+\alpha)}[\log(1/t)]^{\kappa/(1+\alpha)}$ and the optimal rate is $\Theta_\delta(n^{-1/(1+\alpha)}(\log(n))^{\kappa/(1+\alpha)})$. Note that the logarithmic term is in the numerator, in contrast to those of the triangle and Epanechnikov densities. We take $(\alpha,\kappa)=(1/4,1/2)$ in our experiments, and the desired optimal rate should be $\Theta_\delta(n^{-4/5}(\log(n))^{2/5})$.
\end{enumerate}
The $6$ candidate densities are scaled in a way that the resulting scatter plots are roughly within the range. This scaling operation will only act as a vertical shift under the log-log axis and will affect neither the slope estimation in Section~\ref{sec:examine_instance_optimality} nor the comparison between different algorithms in Section~\ref{sec:compare_methods}. We choose the failure probability to be $\delta=0.1$ and evaluate an algorithm by its $0.9$-th empirical quantile of absolute errors. We fix $\mu=0$ due to the shift-invariance of all three algorithms: our estimator \eqref{eq:estimator}, Compton and Valiant's balance-finding estimator \cite{compton2026attainability}, and the constrained asymptotic variance selection (CAVS) estimator \cite{kao2024choosing}. In our estimation procedure, we implement Algorithm~\ref{alg:estimator} with the initial rank fixed to $k=2$. We also adopt a standard cross-fitting strategy in sample splitting: we construct the estimator \eqref{eq:estimator} with the weights from the first split and the mid-summaries from the second split, and then construct \eqref{eq:estimator} again with the weights from the second split and the mid-summaries from the first split, and finally take the average of these two estimators.

\subsection{Examining Instance-Optimality}\label{sec:examine_instance_optimality}
We start by examining whether our Algorithm~\ref{alg:estimator} can achieve the claimed instance-optimality guarantee as stated in Theorem~\ref{thm:main}. For each sample size $n\in\{2^8,2^9,\cdots,2^{16}\}$, we simulate $10^4$ times for each of the $6$ aforementioned densities and calculate the $0.9$-th quantile of absolute estimation errors. When the target instance-optimal rate scales as
\begin{align*}
    O_\delta\left(n^{-\beta}(\log(n))^\eta\right),
\end{align*}
we plot 
\begin{align*}
    y_n=\log_2\left[\frac{(\log(n))^\eta}{\widehat{q}_{0.9}\left(\abs{\widehat{\mu}}\right)}\right]\quad \textnormal{against}\quad x_n=\log_2(n).
\end{align*}
We also fit a linear regression to $\{(x_n,y_n)\}$ for the larger half of the sample sizes $n\in\{2^{12},2^{13},\cdots,2^{16}\}$. When instance-optimality is achieved, the slope estimated from linear regression should be close to $\beta$ -- which is $0.5$ for Gaussian, $1.0$ for uniform, $0.8$ for our selected power-log density, etc.

As can be seen in Figure~\ref{fig:examine_slope}, the estimated slopes are $0.504$ (Gaussian), $1.003$ (uniform), $0.492$ (triangle), $0.675$ (semicircle), $0.493$ (Epanechnikov), and $0.781$ (power-log $(1/4,1/2)$), respectively. Each of them is at most $0.08$ away from its desired value, indicating a high confidence in the instance-optimality guarantee.

\begin{figure}[htb]
    \centering
    \includegraphics[width=1.0\linewidth]{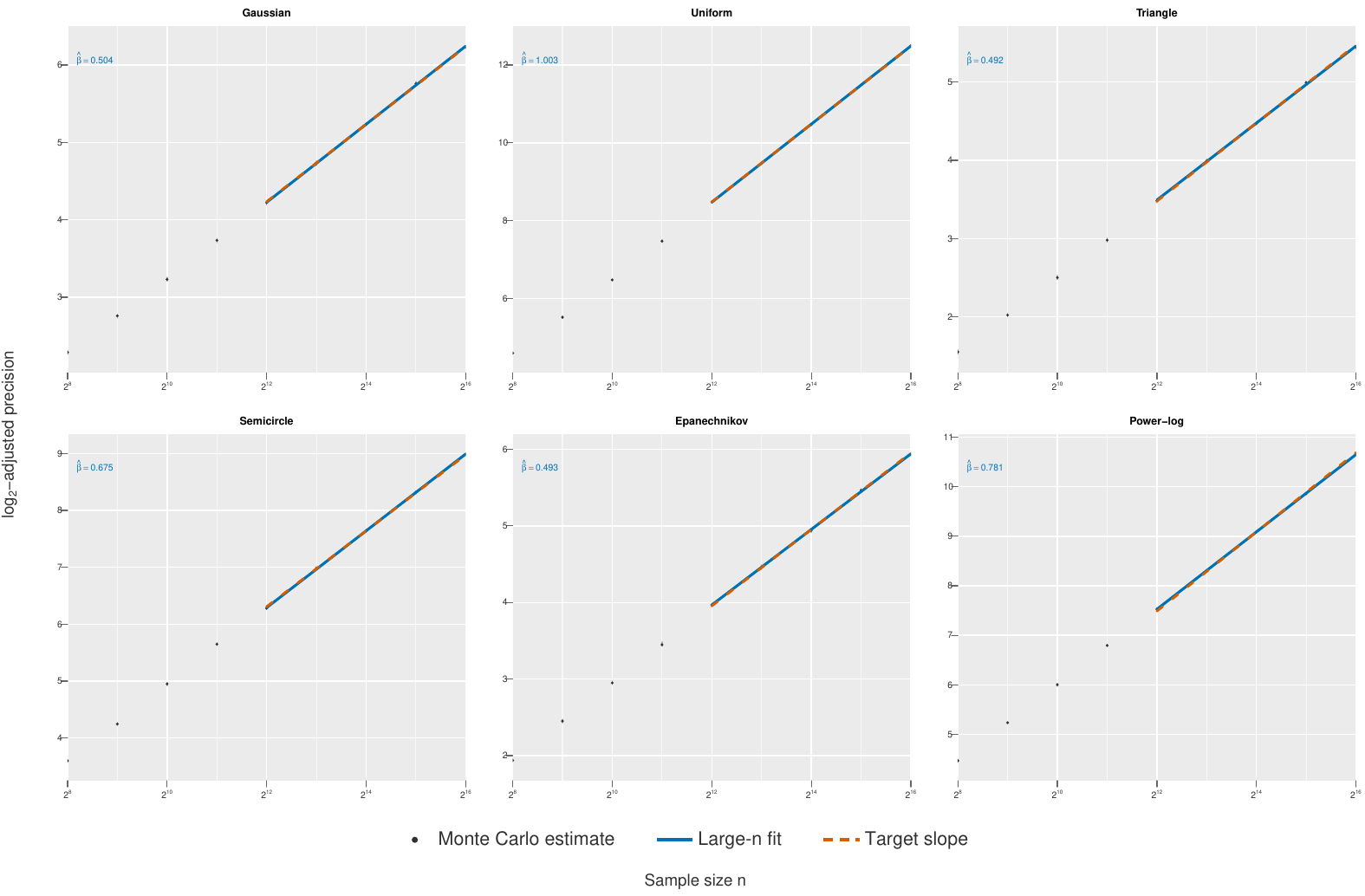}
    \caption{Examination of instance-optimality by regressing the adjusted precision $y_n$ against logarithmic sample size $x_n$. The estimated slopes are $0.504$ (Gaussian), $1.003$ (uniform), $0.492$ (triangle), $0.675$ (semicircle), $0.493$ (Epanechnikov), $0.781$ (power-log $(1/4,1/2)$), respectively.}
    \label{fig:examine_slope}
\end{figure}

\subsection{Cross-Methodology Comparisons}\label{sec:compare_methods}
We also compare our method with the balance-finding estimator in \cite{compton2026attainability} and the CAVS algorithm from \cite{kao2024choosing} (with tuning parameter $\tau_n=\sqrt{\log(4n/200)}$ advised therein), both of which have already been introduced in Section~\ref{sec:related_work}. On the aforementioned $6$ candidate densities, we redo the previous simulations for three estimators and compare the $0.9$-th empirical quantiles of their absolute errors. We divide the estimated error quantiles by a constant multiple of the optimal rate function and then plot these rescaled quantiles against the sample sizes, both in their base-$2$ logarithms. This rescaling step improves the visualization by stabilizing the range of the vertical axis in each of the $6$ sub-plots, and it will not alter the ranking among the three algorithms.

\paragraph{Comparison of rate.}The comparison results are presented in Figure~\ref{fig:comparison}. As the results indicate, no single estimator can dominate the other two in all $6$ cases. While our estimator and Compton and Valiant's balance-finding estimator are stable in all $6$ cases, the CAVS estimator by Kao, Xu, and Zhang exhibits diverging looseness on triangle and Epanechnikov densities that are linear near the support boundaries.

\paragraph{Comparison of computation.} At $n=2^{16}$, the median running time is $0.0016$s (our algorithm), $2.27$s (balance-finding algorithm \cite{compton2026attainability}), and $1.18$s (CAVS algorithm \cite{kao2024choosing}), respectively. Our algorithm appears to be over $700$ times faster than the two previous approaches. This computational distinction reflects the fact that our algorithm runs in $O(\log(n))$ time on sorted samples, while both the balance-finding and the CAVS algorithms require at least $\Omega(n\log(n))$ post-sorting computations.

\begin{figure}[htb]
    \centering
    \includegraphics[width=1.0\linewidth]{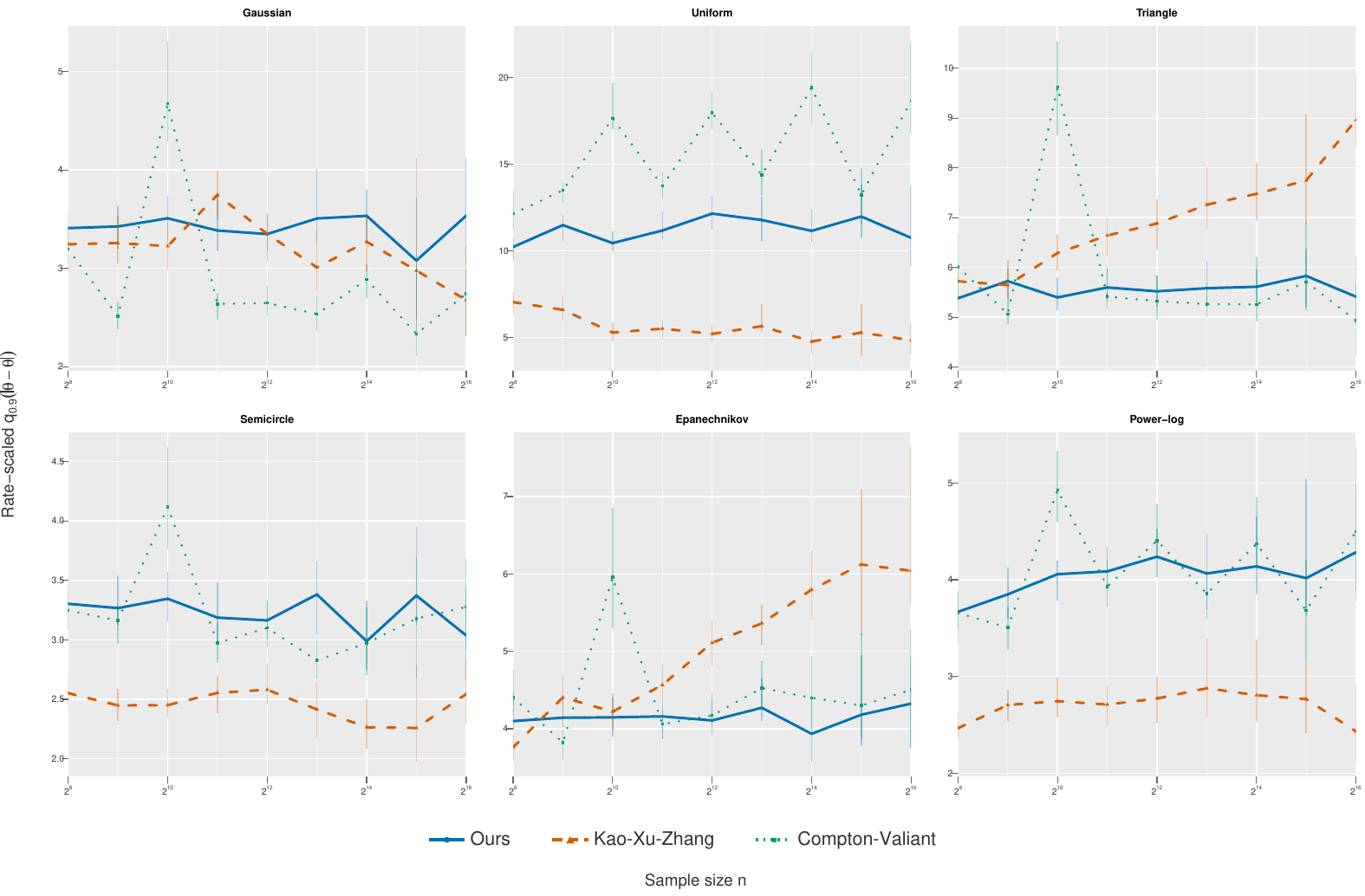}
    \caption{Rescaled $0.9$-th error quantile versus sample size in log-log scale for our Algorithm~\ref{alg:estimator} (blue solid lines), balance-finding algorithm \cite{compton2026attainability} (green dotted lines), and CAVS algorithm \cite{kao2024choosing} (orange dashed lines), respectively.}
    \label{fig:comparison}
\end{figure}

\section{Extension to Densities with Monotone Hazard Rates}\label{sec:mhr}
While the preceding analysis always assumes the log-concavity of $f$, a tempting question is whether the Hellinger-quantile relation and instance-optimal adaptation can be generalized to weaker assumptions. In this section, we show that log-concavity can be relaxed to a monotone\footnote{By convention, monotone means non-decreasing in this context.} hazard rate (MHR). Specifically, we consider the class of densities that are symmetric, unimodal, and MHR:
\begin{align}\label{eq:mhr}
    \cF_\mathsf{SUMHR}\define &\Big\{f\ \textnormal{is a probability density}:\ f(x)=f(-x),\ f(x)\ \textnormal{is non-increasing on}\ [0,+\infty),\nonumber\\
    &\quad h_f(x)=\frac{f(x)}{1-F(x)}\ \textnormal{is non-decreasing on the interior of the support of}\ f \Big\},
\end{align}
where $h_f(x)\define f(x)/(1-F(x))$ is called the hazard function of $f$. The condition \eqref{eq:mhr} is weaker than $\cF_\mathsf{SLC}$ defined in \eqref{eq:slc}. Indeed, for any $f\in\cF_\mathsf{SLC}$, the concavity of $\log f$ implies that
\begin{align*}
    x\mapsto \log f(x+t)-\log f(x)
\end{align*}
is non-increasing in $x$ for every $t\geq 0$, and so is $x\mapsto f(x+t)/f(x)$. Integrating over $t\geq 0$, the function
\begin{align*}
    \frac{1-F(x)}{f(x)}=\int_0^{+\infty} \frac{f(x+t)}{f(x)}\dif t
\end{align*}
is also non-increasing in $x$ in the support interior of $f$. Taking its reciprocal function concludes the MHR property. The reverse direction is not true. For example, a truncated standard Cauchy density
\begin{align*}
    f(x)\,\propto\, \frac{1}{1+x^2}\indi\{\abs{x}\leq A\}
\end{align*}
belongs to $\cF_\mathsf{SLC}$ only when $0<A\leq 1$, but it belongs to $\cF_\mathsf{SUMHR}$ whenever $0<A\leq \tan(\pi/4+1/2)\approx 3.408$. As another example, the quadratically tilted Laplace density
\begin{align*}
    f(x)\,\propto\,\frac{x^2+2\abs{x}+3}{14}e^{-\abs{x}}
\end{align*}
belongs to $\cF_\mathsf{SUMHR}$ but is not log-concave.

Similar to Proposition~\ref{prop:critical_hellinger_by_quantile}, we claim that a Hellinger-quantile relation holds universally over $\cF_\mathsf{SUMHR}$. Specifically, we have:
\begin{Proposition}\label{prop:critical_hellinger_by_quantile_mhr}
For any $p\in(0,1/4]$ and any $f\in\cF_\mathsf{SUMHR}$, we have
\begin{align}\label{eq:critical_hellinger_by_quantile_mhr}
    \frac{(4t^2\wedge 1)p}{2304}\leq\HL^2\left(f_0,f_{2t\cdot\Delta_f(p)}\right)\leq  1\wedge (40t^2+4t)p,
\end{align}
for all $t\geq 0$, where the quantity $\Delta_f(p)$ is defined as in \eqref{eq:delta_f_p}.
\end{Proposition}
The upper bound in \eqref{eq:critical_hellinger_by_quantile_mhr} is identical to the one in \eqref{eq:critical_hellinger_by_quantile}, while its lower bound is not as tight as the one in \eqref{eq:critical_hellinger_by_quantile_mhr} for large $t$. This is because, with a review of the proof structure of Proposition~\ref{prop:critical_hellinger_by_quantile}, it turns out that the only needed assumption for the all-$t$ upper bound and the local small-$t$ lower bound is a non-decreasing hazard rate; however, without the log-concavity of $f$, we cannot claim the log-concavity of the affinity function $r\mapsto \int \sqrt{f_0f_{2r}}$ as in Lemma~\ref{lem:affinity}, which prevents us from obtaining a lower bound that consistently increases with large $t$. Nevertheless, this distinction will not create any barriers to our goal of instance-optimal adaptation since only the upper bound is required to control the estimator's performance; also, we always set $t\asymp 1$ in our proof, so a weaker $t$-dependence in the lower bound does not matter. Therefore, in analogy to proof in the symmetric log-concave setting, we have:
\begin{Corollary}\label{cor:main_mhr}
For any $\delta\in(0,1/3]$, the estimator $\widehat{\mu}_\delta=\widehat{\mu}_\delta(X_{1:n})$ defined in \eqref{eq:estimator} satisfies 
\begin{align}
    \inf_{\mu\in\bbR}\ \Prob_{X_{1:n}\iid f_\mu}\left(\abs{\widehat{\mu}_\delta-\mu}\leq \omega_f\Big(\frac{C\log(1/\delta)}{n}\Big)\right)\geq 1-\delta\quad\textnormal{whenever}\  n\geq C'\log(1/\delta),
\end{align}
for every $f\in\cF_\mathsf{SUMHR}$, where $C$ and $C'$ are universal absolute constants.
\end{Corollary}

\bibliography{ref}
\bibliographystyle{alpha}

\appendix
\section{Proof of Proposition~\ref{prop:critical_hellinger_by_quantile}}\label{sec:proof_quantile_to_hellinger}

\subsection{Preliminary Facts of Symmetric Unimodal Location Families}
Let $f$ be a symmetric unimodal probability density function on $\bbR$. We write $F$ as its cumulative distribution function and let
\begin{align*}
    q(u)=\inf\{x\in\bbR:\ F(x)\geq u\},\quad u\in(0,1)
\end{align*}
be its quantile function. We further define
\begin{align*}
     I_f=\{x\in\bbR:\ f(x)>0\}
\end{align*}
as the support of $f$ and define $I_f^\circ$ as the support interior. We denote by
\begin{align*}
    s(x)=(\log f)'(x),\quad x\in I_f^\circ
\end{align*}
the score function, if it exists at that point (which is actually true for a.e. $x\in I_f^\circ$). We also use the shorthand notation
\begin{align*}
    d(u)=f(q(u)),\quad u\in(0,1)
\end{align*}
to denote the value of $f$ at its $u$-th quantile. We state the following basic facts without proof.
\begin{Lemma}\label{lem:su_preliminary}
    Let $f$ be a symmetric unimodal probability density function on $\bbR$. We have:
    \begin{enumerate}
        \item There exists a possibly infinite constant $c_f\in (0,+\infty]$ such that $I_f^\circ=(-c_f,c_f)$. Moreover, $ I_f^\circ=\{x\in\bbR:\ 0<F(x)<1\}$;
        \item $q(F(x))=x$ for every $x\in I_f^\circ$;
        \item $f'$, $s=(\log f)'$, and $q'=1/(f\circ q)=1/d$ exist/hold for a.e. $x\in I_f^\circ$;
        \item The function $r\mapsto \HL^2(f_0,f_{2r})$ is non-decreasing on $[0,+\infty)$ and strictly increasing on $[0,c_f)$.
        
    \end{enumerate}
\end{Lemma}

\subsection{Structural Results for Symmetric Log-concave Densities}
We first prove that log-concavity implies a monotone hazard rate, i.e., the hazard function $h_f(x)=f(x)/(1-F(x))$ is non-decreasing in the support interior of $f$.
\begin{Lemma}\label{lem:slc_to_sumhr}
For any $f\in\cF_\mathsf{SLC}$, the hazard function $h_f(x)=f(x)/(1-F(x))$ is non-decreasing in the support interior $I^\circ_f$.
\end{Lemma}
\begin{proof}
We start by showing that
\begin{align*}
    x\mapsto \frac{f(x+t)}{f(x)}
\end{align*}
is non-increasing in $x\in I^\circ_f$ for every $t> 0$. It suffices to show that for arbitrary $x_1< x_2$ in the support interior,
\begin{align*}
    f(x_1+t)f(x_2)\geq f(x_1)f(x_2+t).
\end{align*}
Note that
\begin{align*}
    x_1+t=\frac{x_2-x_1}{x_2-x_1+t}x_1+\frac{t}{x_2-x_1+t}(x_2+t).
\end{align*}
Thus, the log-concavity of $f$ gives
\begin{align*}
    f(x_1+t)\geq f(x_1)^{\frac{x_2-x_1}{x_2-x_1+t}}f(x_2+t)^{\frac{t}{x_2-x_1+t}}.
\end{align*}
Similarly,
\begin{align*}
    x_2=\frac{t}{x_2-x_1+t}x_1+\frac{x_2-x_1}{x_2-x_1+t}(x_2+t),
\end{align*}
and thus
\begin{align*}
    f(x_2)\geq f(x_1)^{\frac{t}{x_2-x_1+t}}f(x_2+t)^{\frac{x_2-x_1}{x_2-x_1+t}}.
\end{align*}
Combining the two inequalities gives the desired conclusion. Using that $x\mapsto f(x+t)/f(x)$ is non-increasing, we conclude that its integrated version
\begin{align*}
    x\mapsto\frac{1-F(x)}{f(x)}=\int_0^{+\infty}\frac{f(x+t)}{f(x)}\dif t
\end{align*}
is also non-increasing. Therefore, its reciprocal $h_f$ is non-decreasing, which concludes the proof.
\end{proof}

From now on, we will always assume the weaker condition that $f$ is symmetric, unimodal, and has a monotone hazard rate. Log-concavity is only used as an additional assumption in Lemma~\ref{lem:affinity} and Lemma~\ref{lem:global_h2_lower_bound}.
\begin{Lemma}\label{lem:monotone_reverse_hazard}
Let $f$ be a symmetric unimodal density with a monotone hazard rate. Then, the function $x\mapsto f(x)/F(x)$ is non-increasing in the support interior $I^\circ_f$, and the function $u\mapsto d(u)/u$ is non-increasing on $(0,1)$. Moreover, 
\begin{align}\label{eq:score_reverse_hazard}
    s(x)\leq \frac{f(x)}{F(x)},\quad \textnormal{a.e.}\ x\in I^\circ_f.
\end{align}
\end{Lemma}
\begin{proof}
That $f/F$ is non-increasing in $I^\circ_f$ follows from the fact that $f(-x)/F(-x)=f(x)/(1-F(x))$ due to the symmetry of $f$. By change-of-variables with $x=q(u)$, the function $u\mapsto d(u)/u=f(q(u))/u$ is also non-increasing. Meanwhile, since $f/F$ is non-decreasing, we have
\begin{align}
    0\leq \left(\frac{f}{F}\right)'=\frac{f'F-f^2}{F^2}
\end{align}
for a.e. $x\in I_f^\circ$, which implies $f'/f\leq f/F$ for a.e. $x\in I_f^\circ$. As a result,
\begin{align*}
    s=(\log f)'=\frac{f'}{f}\leq \frac{f}{F}
\end{align*}
for a.e. $x\in I_f^\circ$, which concludes the proof.
\end{proof}

\begin{Lemma}\label{lem:quantile_gap_stability}
Let $f$ be a symmetric unimodal density with a monotone hazard rate. For any $u\in(0,1/4]$, we have
\begin{align}\label{eq:ell_pdp}
    \frac{u}{2d(u)}\leq \ell(u)\leq \frac{u}{d(u)}.
\end{align}
Furthermore, $q(\cdot)$ is $(2/d(u))$-Lipschitz on $[u/2,2u]$. Moreover, for $0<\alpha_1\leq \alpha_2\leq \beta_1\leq \beta_2$ with $\beta_2u\leq 1/2$, if
\begin{align*}
    u_1\in[\alpha_1 u, \alpha_2 u],\quad u_2\in[\beta_1u, \beta_2 u],
\end{align*}
then
\begin{align}\label{eq:difference_between_q}
    \frac{\beta_1-\alpha_2}{1\vee\beta_1}\,\ell(u)\leq q(u_2)-q(u_1)\leq \frac{2(\beta_2-\alpha_1)}{1\wedge \alpha_1}\,\ell(u).
\end{align}
\end{Lemma}
\begin{proof}
Since $u$ and $2u$ are both below $1/2$, $q(u)$ and $q(2u)$ are both on $(-\infty,0]$, on which $f$ is non-decreasing. Therefore, $f(x)\geq f(q(u))=d(u)$ for every $x\in[q(u),q(2u)]$. Also, $f(x)\leq f(q(2u))=d(2u)$. According to Lemma~\ref{lem:monotone_reverse_hazard}, 
\begin{align*}
    \frac{d(2u)}{2u}\leq \frac{d(u)}{u},
\end{align*}
and thus $f(x)\leq d(2u)\leq 2d(u)$ for every $x\in[q(u),q(2u)]$. Therefore,
\begin{align*}
    u=\int_{q(u)}^{q(2u)}f(x)\dif x\in\left[\int_{q(u)}^{q(2u)}d(u)\dif x,\ \int_{q(u)}^{q(2u)}2d(u)\dif x\right]=\left[1,2\right]\cdot d(u)\ell(u),
\end{align*}
which proves the first line.

We use Lemma~\ref{lem:monotone_reverse_hazard} to obtain $d(u)/u\leq d(u/2)/(u/2)$ and thus $d(u/2)\geq d(u)/2$. Note that $d(v)\geq d(u/2)$ for every $v\in[u/2,2u]$. Therefore, for any $u_1,u_2\in[u/2,2u]$,
\begin{align*}
    \abs{q(u_1)-q(u_2)}=\abs*{\int_{u_1}^{u_2}q'(v)\dif v}=\abs*{\int_{u_1}^{u_2}\frac{1}{d(v)}\dif v}\leq \abs{u_1-u_2}\cdot \frac{1}{d(u/2)}\leq \abs{u_1-u_2}\cdot \frac{2}{d(u)},
\end{align*}
which proves the Lipschitzness claim.

For $v\in [\alpha_1u,1/2]$, we have $d(v)\geq d(\alpha_1 u)$. If $\alpha_1\geq 1$, then $d(\alpha_1 u)\geq d(u)$. Otherwise, we use the monotonicity of $u\mapsto d(u)/u$ from Lemma~\ref{lem:monotone_reverse_hazard} to obtain $d(\alpha_1 u)\geq \alpha_1 u\cdot d(u)/u=\alpha_1 d(u)$. Combining the two cases, we obtain $d(v)\geq (1\wedge \alpha_1)d(u)$ for every $v\in [\alpha_1p,1/2]$. Therefore,
\begin{align*}
    q(u_2)-q(u_1)= \int_{u_1}^{u_2}\frac{1}{d(v)}\dif v\leq \frac{(\beta_2-\alpha_1)u}{(1\wedge \alpha_1)d(u)}\leq  \frac{2(\beta_2-\alpha_1)}{(1\wedge \alpha_1)}\ell(u),
\end{align*}
where the last inequality follows from the first inequality of the current lemma that $u/d(u)\leq 2\ell(u)$. Similarly, we have $d(v)\leq (1\vee\beta_1)d(u)$ for every $v\in (0,\beta_1u]$. Therefore,
\begin{align*}
    q(u_2)-q(u_1)= \int_{u_1}^{u_2}\frac{1}{d(v)}\dif v\geq \int_{\alpha_2u}^{\beta_1 u}\frac{1}{d(v)}\dif v\geq \frac{(\beta_1-\alpha_2)u}{(1\vee \beta_1)d(u)}\geq \frac{(\beta_1-\alpha_2)}{(1\vee \beta_1)}\ell(u),
\end{align*}
where we have used $u/d(u)\geq \ell(u)$ in the last inequality.
\end{proof}

\subsection{Proof of the Hellinger Lower Bound}\label{sec:proof_h2_lower_bound}

\begin{Lemma}[Local lower bound]\label{lem:local_h2_lower_bound}
Let $f$ be a symmetric unimodal density with a monotone hazard rate. For $0\leq r\leq \frac{1}{2}\min_{0\leq j\leq J_p}\ell(2^jp)$, we have
\begin{align*}
    \HL^2\left(f_0,f_{2r}\right)\geq \frac{r^2p}{576\Delta_f^2(p)}.
\end{align*}
Consequently, for $0\leq t\leq 1/2$,
\begin{align*}
    \HL^2\left(f_0,f_{2t\Delta_f(p)}\right)\geq\frac{t^2p}{576}.
\end{align*}
\end{Lemma}
\begin{proof}
Define
\begin{align*}
    T_j=F(q(2^jp)+r)-F(q(2^jp)-r).
\end{align*}
Applying \eqref{eq:difference_between_q} with $(\alpha_1,\alpha_2,\beta_1,\beta_2)=(1/2,1/2,1,1)$, we obtain
\begin{align*}
    \frac{1}{2}\ell(u)\leq q(u)-q(u/2)\leq 2\ell(u),\quad 0<u\leq 1/4.
\end{align*}
Therefore,
\begin{align*}
    q(2^jp)-r\geq q(2^{j-1}p)+\frac{1}{2}\ell(2^jp)-r\geq q(2^{j-1}p),\quad j=0,1,2,\cdots,J_p.
\end{align*}
On the other hand, 
\begin{align*}
    q(2^jp)+r=q(2^{j+1}p)-\ell(2^jp)+r\leq q(2^{j+1}p),\quad j=0,1,2,\cdots,J_p,
\end{align*}
and that $2^{j+1}p\leq 2\cdot 1/4\leq 1/2$. As a result, the interval $[q(2^jp)\pm r]$ lies within $(-\infty,0]$. Lemma~\ref{lem:monotone_reverse_hazard} gives $d(2^{j-1}p)\geq d(2^jp)/2$. Therefore, 
\begin{align}\label{eq:tj_rd}
    T_j=\int_{q(2^jp)-r}^{q(2^jp)+r}f(x)\dif x\geq 2r\cdot f(q(2^jp)-r)\geq 2r\cdot f(q(2^{j-1}p))=2rd(2^{j-1}p)\geq rd(2^jp).
\end{align}
Consider the following partition of $\bbR$ into $J_p+2$ disjoint cells:
\begin{align*}
    A_0=(-\infty,q(p)+r],\quad A_j=(q(2^{j-1}p)+r,q(2^jp)+r]\quad (1\leq j\leq J_p),
\end{align*}
and finally
\begin{align*}
    A_{J_p+1}=(q(2^{j_p}p)+r,+\infty).
\end{align*}
For the first cell,
\begin{align*}
    F_0(A_0)-F_{2r}(A_0)=F(q(p)+r)-F(q(p)+r-2r)=T_0.
\end{align*}
Similarly,
\begin{align*}
     F_0(A_j)-F_{2r}(A_j)&=F(q(2^jp)+r)-F(q(2^{j-1}p)+r)-F(q(2^jp)-r)+F(q(2^{j-1}p)-r)\\
     &=T_j-T_{j-1},\quad j=1,2,\cdots,J_m.
\end{align*}
Furthermore, since we have previously proved that $q(2^jp)+r\leq q(2^{j+1}p)$, we have
\begin{align*}
     F_0(A_j)+F_{2r}(A_j)= F(q(2^jp)+r)+F(q(2^jp)-r)\leq F(q(2^{j+1}p))+F(q(2^jp))=3\cdot 2^jp.
\end{align*}
Consider the deterministic channel $\cT:\bbR\to [0:J_p+1]$ given by
\begin{align*}
    \cT(x)=j\quad \textnormal{if}\ x\in A_j.
\end{align*}
For any density $g$, we denote by $\cT\circ g$ the distribution of $\cT(X)$ where $X\sim g$. By the data-processing inequality, we have
\begin{align*}
    \HL^2(f_0,f_{2r})\geq \HL^2(\cT\circ f_0,\cT\circ f_{2r})&= \frac{1}{2}\sum_{j=0}^{J_p+1} \left(\sqrt{F_0(A_j)}-\sqrt{F_{2r}(A_j)}\right)^2\\
    &\geq\frac{1}{4} \sum_{j=0}^{J_p} \frac{\left(F_0(A_j)-F_{2r}(A_j)\right)^2}{F_0(A_j)+F_{2r}(A_j)}\\
    &\geq \frac{1}{12}\left(\frac{T_0}{p}+\sum_{j=1}^{J_p}\frac{T_j-T_{j-1}}{2^jp}\right)\\
    &\geq \frac{1}{144}\sum_{j=0}^{J_p}\frac{T_j^2}{2^jp},
\end{align*}
where we have used Lemma~\ref{lem:difference_hardy} in the last line. Finally, we use \eqref{eq:tj_rd} and then \eqref{eq:ell_pdp} to obtain that
\begin{align*}
    \frac{T_j^2}{2^jp}\geq \frac{r^2d^2(2^jp)}{2^jp}\geq \frac{r^2\cdot 2^jp}{4\ell^2(2^jp)}.
\end{align*}
Therefore, 
\begin{align*}
    \HL^2(f_0,f_{2r})\geq\frac{r^2p}{576}\sum_{j=0}^{J_p}\frac{2^j}{\ell^2(2^jp)}=\frac{r^2p}{576\Delta_f^2(p)},
\end{align*}
which is precisely the first desired inequality. To prove the second inequality, we notice that
\begin{align*}
    \Delta_f(p)=\left(\sum_{j=0}^{J_p}\frac{2^j}{\ell^2(2^jp)}\right)^{-1/2}\leq \frac{\ell(2^jp)}{\sqrt{2^j}},\quad j=0,1,2,\cdots,J_p.
\end{align*}
Therefore, when $0\leq t\leq 1/2$, we will have $t\Delta_f(p)\leq \frac{1}{2}\max_{0\leq j\leq J_p} \ell(2^jp)$, and the second inequality follows from the previous one.
\end{proof}

\begin{Lemma}\label{lem:affinity}
Let $f\in\cF_\mathsf{SLC}$ and $g=\sqrt{f}$. Then, the function
\begin{align*}
    A_f(r)=\int_\bbR \sqrt{g(x)g(x-2r)}\dif x,\quad t\in\bbR
\end{align*}
is even and log-concave. In particular, $A_f(0)=1$ and, for $0<r_0\leq r_1$,
\begin{align*}
    A_f(r_1)\leq A_f(r_0)^{r_1/r_0}.
\end{align*}
\end{Lemma}
\begin{proof}
Evenness follows from the change-of-variables $x\mapsto -x$ and the fact that $g$ is even. To prove log-concavity, choose arbitrary $r_0,r_1\in\bbR$ and define $r_\lambda=(1-\lambda)r_0+\lambda r_1$ for $\lambda\in(0,1)$. Log-concavity of $g$ implies that for any $x_0$ and $x_1$ in $\bbR$,
\begin{align*}
    &\ g\left((1-\lambda)x_0+\lambda x_1\right)g\left((1-\lambda)x_0+\lambda x_1-2r_\lambda\right)\\
    &\geq g\left(x_0\right)^{1-\lambda} g\left(x_1\right)^\lambda g\left(x_0-2r_0\right)^{1-\lambda}g\left( x_1-2r_1\right)^\lambda\\
    &\geq \left[g\left(x_0\right)g\left(x_0-2r_0\right)\right]^{1-\lambda}\left[g\left(x_1\right)g\left( x_1-2r_1\right)\right]^\lambda.
\end{align*}
Therefore, applying the Pr'{e}kopa-Leindler inequality (Lemma~\ref{lem:prekopa_leindler}) yields
\begin{align*}
    A_f(r_\lambda)&=\int g(x)g(x-2r_\lambda)\dif x\\
    &\geq \left(\int g(x)g(x-2r_0)\dif x\right)^{1-\lambda}\left(\int g(x)g(x-2r_1)\dif x\right)^\lambda\\
    &=A_f(r_0)^{1-\lambda}A_f(r_1)^\lambda,
\end{align*}
which proves the log-concavity of $A_f(\cdot)$. It is straightforward that $A_f(0)=\int_\bbR f=1$. For $0<r_0< r_1$, we can write
\begin{align*}
    r_0=(1-r_0/r_1)\cdot 0+(r_0/r_1)\cdot r_1.
\end{align*}
Thus, by the log-concavity of $A_f$, we obtain
\begin{align*}
    A_f(r_0)\geq A_f(0)^{1-r_0/r_1}A_f(r_1)^{r_0/r_1}=A_f(r_1)^{r_0/r_1},
\end{align*}
which is exactly
\begin{align*}
    A_f(r_1)\leq A_f(r_0)^{r_1/r_0}.
\end{align*}
When $0<r_0=r_1$, this inequality also holds trivially.
\end{proof}

\begin{Lemma}[Global lower bound]\label{lem:global_h2_lower_bound}
Let $f\in\cF_\mathsf{SLC}$. For every $t\geq 1/2$, we have
\begin{align*}
    \HL^2\left(f_0,f_{2t\Delta_f(p)}\right)\geq 1-\exp\left(-\frac{tp}{1152}\right).
\end{align*}
\end{Lemma}
\begin{proof}
By Lemma~\ref{lem:affinity}, for $t\geq t_0= 1/2$, we have
\begin{align*}
     \HL^2\left(f_0,f_{2t\Delta_f(p)}\right)=1-A_f(t\Delta_f(m))\geq 1-A_f(t_0\Delta_f(m))^{t/t_0}.
\end{align*}
Meanwhile, by Lemma~\ref{lem:local_h2_lower_bound}, we have
\begin{align*}
    A_f(t_0\Delta_f(p))=1-\HL^2\left(f_0,f_{2t_0\Delta_f(p)}\right)\geq\frac{p}{2304}.
\end{align*}
Chaining the two inequalities, we obtain
\begin{align*}
     \HL^2\left(f_0,f_{2t\Delta_f(p)}\right)\geq 1-\left(1-\frac{p}{2304}\right)^{t/t_0}\geq 1-\exp\left(-\frac{tp}{1152}\right),
\end{align*}
where we have used $e^x\geq 1+x$ in the last inequality.
\end{proof}

\begin{proof}[Proof of Proposition~\ref{prop:critical_hellinger_by_quantile}: lower bound]
It suffices to show that the proposed lower bound is compliant with both pieces in Lemma~\ref{lem:local_h2_lower_bound} and Lemma~\ref{lem:global_h2_lower_bound}. When $0\leq t\leq 1/2$, we have
\begin{align*}
    \frac{t^2p}{576}=\frac{(2t^2\wedge t)p}{1152}\geq (1-1/e)\left(\frac{(2t^2\wedge t)p}{1152}\wedge 1\right).
\end{align*}
When $t\geq 1/2$, we note that $h(x)=1-e^{-x}$ is increasing and concave and thus $h(x)\geq h(1)(x\wedge 1)$, which implies
\begin{align*}
    1-\exp\left(-\frac{tp}{1152}\right)\geq (1-1/e)\left(\frac{tp}{1152}\wedge 1\right)\geq (1-1/e)\left(\frac{(2t^2\wedge t)p}{1152}\wedge 1\right).
\end{align*}
The desired conclusion is thus proved.
\end{proof}

\subsection{Proof of the Hellinger Upper Bound}
\begin{Lemma}\label{lem:l2_score}
Let $f$ be a symmetric unimodal density with a monotone hazard rate. We have
\begin{align*}
    \int_{q(p)}^{-q(p)}s^2(x)f(x)\dif x\leq \frac{16p}{\Delta_f^2(p)}.
\end{align*}
\end{Lemma}
\begin{proof}
Note that $q(2^jp)<0$ for all $j=0,1,\cdots,J_p$ since $2^{j_p}p\leq 1/4<1/2$. On $[q(p),0]$, we have $s(x)\geq 0$ almost everywhere. For $0\leq j\leq J_p$, the interval
\begin{align*}
    [q(2^jp),q(2^{j+1}p)]
\end{align*}
has probability $2^jp$ under $f$. Almost everywhere in this interval, Lemma~\ref{lem:monotone_reverse_hazard} gives
\begin{align*}
    0\leq s(x)\leq \frac{f(x)}{F(x)}.
\end{align*}
Meanwhile, in this interval, we also have
\begin{align*}
    F(x)\geq 2^jp,\quad f(x)\leq d(2^{j+1}p)\leq 2d(2^jp),
\end{align*}
where the last inequality comes from the fact that $d(u)/u$ is non-increasing according to Lemma~\ref{lem:monotone_reverse_hazard}. Therefore,
\begin{align*}
    \int_{q(2^jp)}^{q(2^{j+1}p)}s^2(x)f(x)\dif x\leq \ell(2^jp)\cdot \frac{8d^3(2^jp)}{(2^jp)^2}\leq \ell(2^jp)\frac{8}{(2^jp)^2}\cdot \frac{(2^jp)^3}{\ell^3(2^jp)}=8\frac{2^jp}{\ell^2(2^jp)},
\end{align*}
where the last inequality is by $\ell(u)\leq u/d(u)$ according to \eqref{eq:ell_pdp} of Lemma~\ref{lem:quantile_gap_stability}. On the final interval
\begin{align*}
    [q(2^{J_p}p),0],
\end{align*}
we have
\begin{align*}
    0\leq s(x)\leq \frac{f(x)}{F(x)}\leq \frac{f(q(2^{J_p}p))}{F(q(2^{J_p}p))}= \frac{d(2^{J_p}p)}{2^{J_p}p},
\end{align*}
where the second inequality follows from Lemma~\ref{lem:monotone_reverse_hazard} that $f/F$ is non-increasing. Furthermore, on this interval, we have $f(x)\leq f(0)=1/2$. Therefore,
\begin{align*}
    \int_{q(2^{J_p}p)}^{0}s^2(x)f(x)\dif x\leq \frac{d^2(2^{J_p}p)}{2(2^{J_p}p)^2}\leq \frac{1}{2\ell^2(2^{J_p}p)}\leq \frac{4(2^{J_p}p)}{\ell^2(2^{J_p}p)},
\end{align*}
where the last inequality follows from $2^{J_p}p\geq 1/8$ by our choice of $J_p$. Combining the above pieces, we obtain
\begin{align*}
     \int_{q(p)}^{-q(p)}s^2(x)f(x)\dif x=2\int_{q(p)}^{0}s^2(x)f(x)\dif x\leq 16\sum_{j=0}^{J_p}\frac{2^jp}{\ell^2(2^jp)}=\frac{16p}{\Delta_f^2(p)},
\end{align*}
which concludes the proof.
\end{proof}

\begin{Lemma}[Local upper bound]\label{lem:local_h2_upper_bound}
Let $f$ be a symmetric unimodal density with a monotone hazard rate. If $0\leq r\leq \ell(p)/2$, then
\begin{align*}
    \HL^2\left(f_0,f_{2r}\right)\leq \frac{32r^2p}{\Delta_f^2(p)}+4rd(p).
\end{align*}
Consequently, for $0\leq t\leq 1/2$,
\begin{align*}
    \HL^2\left(f_0,f_{2t\Delta_f(p)}\right)\leq (32t^2+4t)p.
\end{align*}
\end{Lemma}

\begin{proof}
Write $g=\sqrt{f}$ for brevity. Consider the interval:
\begin{align*}
    I_1=[q(p)+2r,-q(p)].
\end{align*}
Whenever it is not empty, for any $x\in I_1$, the interval $[x-2r,x]$ lies within $[q(p),-q(p)]$, on which $f$ is strictly positive. We have, by Cauchy--Schwarz, that
\begin{align*}
    \left(g(x)-g(x-2r)\right)^2\leq 2r\int_{x}^{x-2r}(g')^2(y)\dif y&=2r\int_{x-2r}^{x}s^2(y)f(y)\dif y.
\end{align*}
Then, integrating over $x\in I_1$ and applying Fubini's theorem, we obtain
\begin{align*}
    \int_{I_1} \left(g(x)-g(x-2r)\right)^2\dif x&\leq 2r\int_{I_1}\int_{x-2r}^{x}s^2(y)f(y)\dif y\dif x\\
    &=4r^2\int_{I_1}s^2(y)f(y)\dif y\\
    &\leq 4r^2\int_{q(p)}^{-q(p)}s^2(y)f(y)\dif y\\
    &\leq \frac{64r^2p}{\Delta_f^2(p)},
\end{align*}
where the last line follows from Lemma~\ref{lem:l2_score}. For any $x$ in the left tail:
\begin{align*}
    I_2=(-\infty,q(p)+2r],
\end{align*}
the assumption $0\leq r\leq \ell(p)/2$ implies
\begin{align*}
    x\leq q(p)+2r\leq q(2p)<0.
\end{align*}
Thus $f(x)\geq f(x-2r)$ for every $x\in (-\infty,q(p)+2r]$, and therefore
\begin{align*}
     \left(g(x)-g(x-2r)\right)^2\leq f(x)-f(x-2r).
\end{align*}
Consequently,
\begin{align*}
     \int_{I_2} \left(g(x)-g(x-2r)\right)^2\dif x&\leq \int_{I_2} f(x)-f(x-2r)\dif x\\
     &=F(q(p)+2r)-F(q(p))\\
     &\leq 2r f(q(p)+2r)\\
     &\leq 2r f(q(2p))\\
     &=2r d(2p)\\
     &\leq 4rd(p),
\end{align*}
where the last line follows from Lemma~\ref{lem:monotone_reverse_hazard}. By symmetry, the same bound holds for the right tail $I_3=[-q(p),+\infty)$. Combining the three pieces, we obtain
\begin{align*}
    \HL^2\left(f_0,f_{2r}\right)=\frac{1}{2}\left(\int_{I_1}+\int_{I_2}+\int_{I_3}\right) \left(g(x)-g(x-2r)\right)^2\dif x\leq \frac{32r^2p}{\Delta_f^2(p)}+4rd(p).
\end{align*}
Finally, we set $r=t\Delta_f(p)$. Note that for $0\leq t\leq 1/2$, 
\begin{align*}
    t\Delta_f(p)\leq \frac{1}{2}\cdot \frac{1}{\sqrt{1/\ell^2(p)}}=\frac{\ell(p)}{2},
\end{align*}
so the condition of the preceding upper bound is satisfied. Also, we have
\begin{align*}
    d(p)\leq \frac{p}{\ell(p)}\leq p\cdot\sqrt{\sum_{j=0}^{J_p}\frac{2^j}{\ell^2(2^jp)}}= p\cdot \frac{1}{\Delta_f(p)}
\end{align*}
Therefore, we have
\begin{align*}
     \HL^2\left(f_0,f_{2t\Delta_f(p)}\right)&\leq \frac{32t^2\Delta_f^2(p)}{\Delta_f^2(p)}+4t\Delta_f(p)\cdot \frac{p}{\Delta_f(p)}\\
     &=(32t^2+4t)p,
\end{align*}
which concludes the proof.
\end{proof}
    
\begin{proof}[Proof of Proposition~\ref{prop:critical_hellinger_by_quantile}: upper bound]
The desired bound is automatically true when $t=0$. For each $t> 0$, let 
\begin{align*}
    N=1\vee \ceil*{2t}\leq 1+2t.
\end{align*}
Then,
\begin{align*}
    \frac{t}{N}\leq \frac{1}{2}.
\end{align*}
Since $\HL(\cdot,\cdot)$ is a distance, we apply the triangle inequality and Lemma~\ref{lem:local_h2_upper_bound} to obtain
\begin{align*}
    \HL(f_0,f_{2t\Delta_f(p)})&\leq \sum_{i=1}^N\HL\left(f_{\frac{2(i-1)}{N}t\Delta_f(p)},f_{\frac{2i}{N}t\Delta_f(p)}\right)\\
    &=N\cdot  \HL(f_0,f_{\frac{2t}{N}\Delta_f(p)})\\
    &\leq N\cdot \sqrt{\left[32(t/N)^2+4(t/N)\right]p}.
\end{align*}
Therefore, we have
\begin{align*}
     \HL^2(f_0,f_{2t\Delta_f(p)})&\leq N^2\cdot (32t^2/N^2+4t/N)p\\
     &=(32t^2+4tN)p\\
     &\leq (32t^2+4t(1+2t))p\\
     &=(40t^2+4t)p,
\end{align*} 
which concludes the proof.
\end{proof}
\section{Remaining Proofs of the Estimator's Performance}\label{sec:proof_upper_bound}

\subsection{Proof of Lemma~\ref{lem:efron_stein_coordinatewise_lipschitz}}
\begin{Lemma}\label{lem:exponential_lipschitz_moment}
Let $E\sim \mathrm{Exp}(1)$ and $g:\bbR_+\to \bbR$ be an $L$-Lipschitz function for $L\in[0,+\infty)$. Then, for every integer $j\geq 1$,
\begin{align*}
    \E\abs{g(E)-\E[g(E)]}^j\leq j!L^j.
\end{align*}
\end{Lemma}
\begin{proof}
Let $E'$ be an independent copy of $E$. Conditioned on $E$,
\begin{align*}
    g(E)-\E[g(E)]=\E_{E'}[g(E-g(E')].
\end{align*}
By Jensen's inequality and the convexity of $x\mapsto x^j$,
\begin{align*}
    \abs{g(E)-\E[g(E)]}^j=\abs{\E_{E'}[g(E-g(E')] }^j\leq \E_{E'}\abs{g(E)-g(E')}^j\leq L^j\E_{E'}\abs{E-E'}^j.
\end{align*}
Taking expectation over $E\sim\mathrm{Exp}(1)$, we obtain
\begin{align*}
    \E\abs{g(E)-\E[g(E)]}^j\leq L^j\E_{E,E'}\abs{E-E'}^j=L^j\E_{Z\sim \mathrm{Lap}(0,1)}\abs{Z}^j=j!L^j.
\end{align*}
\end{proof}

\begin{proof}[Proof of Lemma~\ref{lem:efron_stein_coordinatewise_lipschitz}]
We first observe that for any $u,v\in\bbR_+^{N}$,
\begin{align*}
    \abs{H(u)-H(v)}&\leq \sum_{i=1}^N \abs{H((u_1,\cdots,u_{i-1},v_i,\cdots,v_N))-H((u_1,\cdots,u_i,v_{i+1},\cdots,v_N)) }\\
    &\leq \sum_{i=1}^N L_i \abs{u_i-v_i}.
\end{align*}
If all $L_i$ are zero, then the above inequality shows that $H$ is constant, and the desired bound trivially holds. From now on, we assume $L_i$ are not all zero, which means $\max_i L_i>0$ and $\sum_i L_i^2>0$.

Define the filtration
\begin{align*}
    \cF_0=\{\emptyset,\Omega\},\quad \cF_i=\sigma(E_{1:i}),\quad i=1,2,\cdots,N.
\end{align*}
Let $G_i=\E[H(E)\mid \cF_i]$ be the Doob martingale and let $D_i=G_i-G_{i-1}$ be the martingale difference. We then have
\begin{align*}
    \E[D_i\mid \cF_{i-1}]=0,\quad H(E)-\E[H(E)]=\sum_{i=1}^N D_i.
\end{align*}
We then show that each $D_i$, conditioned on the past, is an $L_i$-Lipschitz function of one $\mathrm{Exp}(1)$ variable. Fix any values of $E_{1:(i-1)}$, let
\begin{align*}
    g_i(u)=\E[H(E_1,E_2,\cdots,E_{i-1},u,E_{i+1},\cdots,E_N)\mid E_1,E_2,\cdots,E_{i-1}]
\end{align*}
Then, $M_i=g_i(E_i)$ and $M_{i-1}=\E[g_i(E)\mid \cF_{i-1}]$. Hence
\begin{align*}
    D_i=g_i(E_i)-\E[g_i(E)\mid \cF_{i-1}].
\end{align*}
Note that for arbitrary $u,v\in\bbR_+^N$
\begin{align*}
    \abs{g_i(u)-g_i(v)}\leq \E[\abs{H(\cdots,u,\cdots)-H(\cdots,v,\cdots)}\mid \cF_{i-1}]\leq L_i\abs{u_i-v_i}.
\end{align*}
Thus, conditioned on $\cF_{i-1}$, $D_i$ is an $L_i$-Lipschitz function of $E_i$. Applying Lemma~\ref{lem:exponential_lipschitz_moment} to $D_i$, we obtain
\begin{align*}
    \E[\abs{D_i}^j\mid \cF_{i-1}]\leq j!L_i^j.
\end{align*}
If $L_i>0$, for $0\leq \theta<1/L_i$,
\begin{align*}
    \E[e^{\theta D_i}\mid \cF_{i-1}]=\sum_{j=0}^{\infty}\frac{\theta^j\E[D_i^j\mid \cF_{i-1}]}{j!}\leq 1+\sum_{j= 2}^\infty\frac{\theta^j L_i^j}{j!}\leq 1+\sum_{j=2}^\infty (\theta L_i)^j &=1+\frac{\theta^2L_i^2}{1-\theta L_i}\\
    &\leq \exp\left(\frac{\theta^2L_i^2}{1-\theta L_i}\right).
\end{align*}
In particular,
\begin{align*}
    \E[e^{\theta D_i}\mid\cF_i]\leq \exp\left(\frac{\theta^2L_i^2}{1-\theta\max_i L_i}\right),\quad 0\leq \theta<1/(\max_iL_i).
\end{align*}
The same bound holds trivially when $L_i=0$, and applies to $-D_i$ in the same way. We then have
\begin{align*}
    \E[e^{\theta(H(E)-\E[H(E)])}]=\E[e^{\theta \sum_{i=1}^N D_i}]&=\E[e^{\theta\sum_{i=1}^{N-1}D_i}\E[e^{\theta D_N}\mid \cF_{N-1}]]\\
    &\leq \exp\left(\frac{\theta^2L_i^2}{1-\theta\max_i L_i}\right)\E[e^{\theta\sum_{i=1}^{N-1}D_i}]\\
    &\leq \cdots\\
    &\leq \exp\left(\frac{\theta^2\sum_{i=1}^NL_i^2}{1-\theta\max_i L_i}\right),\quad 0\leq \theta\leq 1/(\max_i L_i).
\end{align*}
For every $s>0$ and $0\leq \theta\leq 1/(\max_iL_i)$, Markov inequality yields
\begin{align*}
    \Prob\left(H(E)-\E[H(E)]> s\right)\leq \exp\left(-\theta s+\frac{\theta^2\sum_iL_i^2}{1-\theta\max_i L_i}\right).
\end{align*}
Fix any $t>0$, choose $s=2\sqrt{t\sum_iL_i^2}+t\max_i L_i>0$ and
\begin{align*}
    \theta=\frac{\sqrt{t}}{\sqrt{\sum_i L_i^2}+\sqrt{t}\max_iL_i}\in [0,1/(\max_iL_i)).
\end{align*}
With these choices, 
\begin{align*}
    \frac{\theta^2\sum_iL_i^2}{1-\theta\max_iL_i}=\frac{t\sqrt{\sum_i L_i^2}}{\sqrt{\sum_iL_i^2}+\sqrt{t}\max_iL_i}
\end{align*}
and
\begin{align*}
    \theta s=\frac{t(2\sqrt{\sum_i L_i^2}+\max_i L_i)}{\sqrt{\sum_iL_i^2}+\sqrt{t}\max_iL_i},
\end{align*}
and therefore
\begin{align*}
   -\theta s +\frac{\theta^2\sum_iL_i^2}{1-\theta\max_iL_i}=-t.
\end{align*}
We have thus obtained
\begin{align*}
    \Prob\left(H(E)-\E[H(E)]> 2\sqrt{t\sum_iL_i^2}+t\max_i L_i\right)\leq e^{-t}.
\end{align*}
The desired conclusion follows from repeating the above procedure for $-H(E)$ and taking the union bound.
\end{proof}

\subsection{Proof of Lemma~\ref{lem:truncated_lipschitz}}\label{sec:proof_truncated_lipschitz}
We use the shorthand notations $d(u)\define f(q(u))$ and
\begin{align*}
    V^\pm_j(E)\define \frac{0.9(2^jk)\vee S^\pm_{2^jk}\wedge 1.1(2^jk)}{0.9(m+1)\vee S_{m+1}\wedge 1.1(m+1)}.
\end{align*}

\begin{Lemma}\label{lem:qv_lipschitz}
If two vectors $E,E'\in\bbR^{m+1}$ differ only in the $i$-th coordinate, then
     \begin{align*}
        \abs{q(V_j^-(E))-q(V_j^-(E'))}\leq 
        \begin{dcases}
            \frac{4}{(m+1)d(2^jp)}\abs{E_i-E'_i},\quad &i\leq 2^jk;\\
            \frac{4(2^jp)}{(m+1)d(2^jp)}\abs{E_i-E'_i},\quad &i> 2^jk;
        \end{dcases}
    \end{align*}
    and
    \begin{align*}
        \abs{q(V_j^+(E))-q(V_j^+(E'))}\leq 
        \begin{dcases}
            \frac{4}{(m+1)d(2^jp)}\abs{E_i-E'_i},\quad &m+2-i\leq 2^jk;\\
            \frac{4(2^jp)}{(m+1)d(2^jp)}\abs{E_i-E'_i},\quad &m+2-i> 2^jk.
        \end{dcases}
    \end{align*}
\end{Lemma}
\begin{proof}
We first inspect the Lipschitzness of the inner argument $V^\pm_j$. Consider $E$ and $E'$ that differ only in their $i$-th argument. For arbitrary positive numbers $a,a',b,b'$, it holds that
\begin{align*}
    \abs*{\frac{a}{b}-\frac{a'}{b'}}=\frac{\abs{ab'-ab+ab-a'b}}{bb'}\leq \frac{a\abs{b-b'}}{bb'}+\frac{\abs{a-a'}}{b'}.
\end{align*}
We now set 
\begin{align*}
    a=0.9(2^jk)\vee S^\pm_{2^jk}(E)\wedge 1.1(2^jk),\quad b=0.9(m+1)\vee S_{m+1}(E)\wedge 1.1(m+1),
\end{align*}
and let $a'$ and $b'$ be their counterparts with $E$ replaced by $E'$. The truncation ensures that both denominators are at least $0.9(m+1)$. Therefore, we have
\begin{align*}
    \abs{V^\pm_j(E)-V^\pm_j(E')}&\leq \frac{(0.9(2^jk)\vee S^\pm_{2^jk}(E)\wedge 1.1(2^jk))}{0.9^2(m+1)^2}\abs{E_i-E'_i}\\
    &\ +\frac{\abs{(0.9(2^jk)\vee S^\pm_{2^jk}(E)\wedge 1.1(2^jk))-(0.9(2^jk)\vee S^\pm_{2^jk}(E')\wedge 1.1(2^jk))}}{0.9(m+1)}\\
    &\leq \frac{1.1(2^jk)}{0.9^2(m+1)^2}\abs{E_i-E'_i}+\frac{\abs{S^\pm_{2^jk}(E)-S^\pm_{2^jk}(E')}}{0.9(m+1)}.
\end{align*}
When $i\leq 2^jk$, $E_i$ (resp. $E'_i$) is included in $S^-_{2^jk}(E)$ (resp. $S^-_{2^jk}(E')$), and thus $S^-_{2^jk}(E)-S^-_{2^jk}(E')=E_i-E'_i$. We have
\begin{align*}
    \abs{V^\pm_j(E)-V^-_j(E')}\leq \left(\frac{1.1k_j}{0.9^2(m+1)^2}+\frac{1}{0.9(m+1)}\right)\abs{E_i-E'_i}\leq \frac{2}{m+1}\abs{E_i-E'_i}.
\end{align*}
When $i>2^jk$, $E_i$ (resp. $E'_i$) is not included in $S^-_{2^jk}(E)$ (resp. $S^-_{2^jk}(E')$), and thus $S^-_{2^jk}(E)-S^-_{2^jk}(E')=0$. We have
\begin{align*}
    \abs{V^\pm_j(E)-V^-_j(E')}\leq \frac{1.1(2^jk)}{0.9^2(m+1)^2}\abs{E_i-E'_i}\leq \frac{2(2^jp)}{m+1}\abs{E_i-E'_i}.
\end{align*}
Meanwhile, the truncation ensures that $V^-_j$ is always within $[2^{j-1}p, 2^{j+1}p]$, on which $q$ is $2/d(2^jp)$-Lipschitz. The conclusion for $q(V^-_j)$ is thus proved. The conclusion for $q(V^+_j)$ holds similarly by replacing $i$ with $m+2-i$. 
\end{proof}
\begin{proof}[Proof of Lemma~\ref{lem:truncated_lipschitz}]
Since
\begin{align*}
    Y_\lambda=\sum_{j=0}^{J_p}\lambda_j\frac{q(V^-_j)-q(V^+_j)}{2},
\end{align*}
if $E,E'\in\bbR^{m+1}$ only differs in the $i$-th coordinate, we have by Lemma~\ref{lem:qv_lipschitz} and triangle inequality that
\begin{align*}
    \abs{Y_\lambda(E)-Y_\lambda(E')}&\leq \frac{2}{m+1}\left(\sum_{j:2^jk\geq i}\frac{\lambda_j}{d(2^jp)}+\sum_{j:2^jk\geq \overline{i}}\frac{\lambda_j}{d(2^jp)}\right)\abs{E_i-E'_i}\nonumber\\
        &\ + \frac{2}{m+1}\left(\sum_{j:2^jk< i}\frac{\lambda_j(2^jp)}{d(2^jp)}+\sum_{j:2^jk< \overline{i}}\frac{\lambda_j(2^jp)}{d(2^jp)}\right)\abs{E_i-E'_i},
\end{align*}
where $\overline{i}=m+2-i$. This proves Lemma~\ref{lem:truncated_lipschitz}.
\end{proof}

\subsection{Proof of Lemma~\ref{lem:dyadic_ineq}}
\begin{proof}[Proof of Lemma~\ref{lem:dyadic_ineq}]
For the first inequality, we notice that
\begin{align*}
     \sum_{i=1}^N\left(\sum_{0\leq j\leq J:\,2^jk\geq i} b_j\right)^2=\sum_{j,l}b_j b_l\sum_{i=1}^N \indi\{i\leq 2^jk\}\indi\{i\leq 2^lk\}&=k\sum_{j,l}(2^j\wedge 2^l)b_j b_l\\
     &=k\sum_{j,l}2^{-\abs{j-l}/2}\sqrt{2^jb_j^2}\sqrt{2^l b_l^2}.
\end{align*}
Then, the desired conclusion follows from $(2^{-\abs{j-l}/2})_{j,l=1}^N\preceq (3+2\sqrt{2})I_N$, which comes from setting $\rho=1/\sqrt{2}$ in Lemma~\ref{lem:toeplitz_sandwich} of Appendix~\ref{sec:aux}.

For the second inequality, positivity gives
\begin{align*}
    \sum_{i=1}^N\left(\sum_{j\leq J:2^jk<i}\frac{2^jb_j}{N}\right)^2\leq N\left(\sum_{j=0}^J\frac{2^jb_j}{N}\right)^2=\frac{1}{N}\left(\sum_{j=0}^J2^jb_j\right)^2.
\end{align*}
Cauchy--Schwarz then gives
\begin{align*}
    \frac{1}{N}\left(\sum_{j=0}^J2^jb_j\right)^2\leq \frac{1}{N}\left(\sum_{j=0}^J 2^j\right)\left(\sum_{j=0}^J 2^j b_j^2\right)=\frac{2^{J+1}-1}{N}\left(\sum_{j=0}^J 2^j b_j^2\right)\leq 2p_\star \left(\sum_{j=0}^J 2^j b_j^2\right),
\end{align*}
which concludes the proof. 
\end{proof}

\subsection{Proof of Lemma~\ref{lem:truncated_mean_and_variance}}
\begin{proof}[Proof of Lemma~\ref{lem:truncated_mean_and_variance}]

For the Lipschitz constants given in Lemma~\ref{lem:truncated_lipschitz}, we can bound their maximum by
\begin{align*}
    \max_i L_i\leq \frac{4}{m+1}\sum_{j=0}^{J_p}\frac{\lambda_j}{f(q(2^jp))}\leq \frac{8}{k}\sum_{j=0}^{J_p}\frac{\lambda_j\ell(2^jp)}{2^j},
\end{align*}
where we have used $f(q(u))\geq u/(2\ell(u))$ from Lemma~\ref{lem:quantile_gap_stability} and $p=k/(m+1)$. For their sum of squares, we have
\begin{align*}
    &\ \sum_{i=1}^{m+1}L_i^2\\
    &\lesssim \frac{4}{(m+1)^2}\sum_{i=1}^{m+1}\left(\sum_{j: 2^jk\geq i}\frac{\lambda_j}{f(q(2^jp))}+\sum_{j: 2^jk\geq m+2-i}\frac{\lambda_j}{f(q(2^jp))}+\sum_{j: 2^jk< i}\frac{\lambda_j2^jp}{f(q(2^jp))}+\sum_{j: 2^jk< m+2-i}\frac{\lambda_j2^jp}{f(q(2^jp))}\right)^2\\
    &\leq \frac{16}{(m+1)^2}\sum_{i=1}^{m+1}\left(\sum_{j: 2^jk\geq i}\frac{\lambda_j}{f(q(2^jp))}\right)^2+\frac{16}{(m+1)^2}\sum_{i=1}^{m+1}\left(\sum_{j: 2^jk\geq m+2-i}\frac{\lambda_j}{f(q(2^jp))}\right)^2\\
    &\ +\frac{16}{(m+1)^2}\sum_{i=1}^{m+1}\left(\sum_{j: 2^jk< i}\frac{\lambda_j2^jp}{f(q(2^jp))}\right)^2+\frac{16}{(m+1)^2}\sum_{i=1}^{m+1}\left(\sum_{j: 2^jk< m+2-i}\frac{\lambda_j2^jp}{f(q(2^jp))}\right)^2\\
    &\define T_1+T_1'+T_2+T_2',
\end{align*}
where we have used $(a+b+c+d)^2\leq 4(a^2+b^2+c^2+d^2)$. Setting $N=m+1$ and $b_j=\lambda_j/f(q(2^jp))$ in the first inequality of Lemma~\ref{lem:dyadic_ineq}, we have
\begin{align*}
    T_1\leq \frac{16k}{(m+1)^2}\cdot (3+2\sqrt{2})\sum_{j=0}^{J_p}2^j  \frac{\lambda_j^2}{f^2(q(2^jp))}.
\end{align*}
The same bound holds for $T_1'$ by replacing the sum over $i$ with the sum over $\overline{i}=m+2-i$. Also, by the second inequality of Lemma~\ref{lem:dyadic_ineq} with $p_\star=1/4$, we have
\begin{align*}
    T_2\leq \frac{16k}{(m+1)^2}\cdot (2/4)\sum_{j=0}^{J_p}2^j  \frac{\lambda_j^2}{f^2q((2^jp))}.
\end{align*}
The same bound holds for $T_2'$. Combining the four pieces, we obtain
\begin{align*}
    \sum_{i=1}^{m+1}L_i^2\leq \frac{256k}{(m+1)^2}\sum_{j=0}^{J_p}\frac{ 2^j\lambda_j^2}{f^2(q(2^jp))}.
\end{align*}
Finally, we apply $f(q(2^jp))\geq (2^jp)/(2\ell(2^jp))$ from Lemma~\ref{lem:quantile_gap_stability} to obtain
\begin{align*}
    \sum_{i=1}^{m+1}L_i^2\leq \frac{256k}{(m+1)^2}\sum_{j=0}^{J_p}\frac{4\cdot 2^j\lambda_j^2 \ell^2(2^jp)}{(2^jp)^2}=\frac{1024}{k}\sum_{j=0}^{J_p}\frac{\lambda_j^2 \ell^2(p_{m,j})}{2^j},
\end{align*}
which concludes the proof.
\end{proof}

\subsection{Proof of Lemma~\ref{lem:gap_and_weight_stable}}\label{sec:proof_gap_and_weight_stable}
\begin{Lemma}\label{lem:good_event_high_prob}
For $\delta\in(0,1/3]$, $k=\ceil{720\log(1/\delta)}$, and $m\geq 4 k$, the event
\begin{align*}
    \cE(\cD_m)=\left\{\abs*{\frac{S_{m+1}}{m+1}-1}\leq \frac{1}{10},\ \abs*{\frac{S^{\pm}_{2^j k}}{2^jk}-1}\leq \frac{1}{10},\ j=0,1,2,\cdots, J_p+1\right\}
\end{align*}
specified by \eqref{eq:good_event} holds with a probability of at least $1-\delta/3$.  
\end{Lemma}
\begin{proof}
According to Lemma~\ref{lem:gamma_concentration}, the event $\cE$ holds with a failure probability of at most
\begin{align*}
    2e^{-c(m+1)}+4\sum_{j=0}^{J_p+1}e^{-c2^jk}\leq 2e^{-c(m+1)}+4\sum_{j=0}^{+\infty}e^{-c(j+1)k}\leq 2e^{-c(m+1)}+\frac{4e^{-ck}}{1-e^{-ck}},
\end{align*}
where $c=\min\{0.1-\log(1.1),-0.1-\log(0.9)\}\approx 0.0046$.
For $\delta\in(0,1/3]$, $k=\ceil{720\log(1/\delta)}$, and $m\geq 4k$, the probability is bounded by
\begin{align*}
    2e^{-4ck}+\frac{4e^{-ck}}{1-e^{-ck}}\leq 2e^{-2880c\log(1/\delta)}+\frac{4e^{-720c\log(1/\delta)}}{1-e^{-720c\log(1/\delta)}}\leq \delta/3,
\end{align*}
which concludes the proof.
\end{proof}

\begin{proof}[Proof of Lemma~\ref{lem:gap_and_weight_stable}]
We represent the hold-out data $\cD'_m=\{X'_{1:m}\}$ using independent $\mathrm{Exp}(1)$ variables $E'_{1:(m+1)}$. Consider the event
\begin{align*}
    \cE'(\cD'_m)=\left\{\abs*{\frac{S_{m+1}(E')}{m+1}-1}\leq \frac{1}{10},\ \abs*{\frac{S^{\pm}_{2^j k}(E')}{2^j k}-1}\leq \frac{1}{10},\ j=0,1,2,\cdots, J_p+1\right\}
\end{align*}
regarding the hold-out data in analogy to \eqref{eq:good_event}.
By Lemma~\ref{lem:good_event_high_prob}, this event also holds with probability at least $1-\delta/3$.
On this event, we have
\begin{align*}
    \frac{9}{11}(2^jp)\leq \frac{S^\pm_{2^jk}}{S_{m+1}}\leq \frac{11}{9}(2^jp),\quad \frac{2\cdot9}{11}(2^jp)\leq \frac{S^\pm_{2^{j+1}k}}{S_{m+1}}\leq \frac{2\cdot 11}{9}(2^jp),\quad j=0,1,2,\cdots,J_p.
\end{align*}
Note that $2\cdot (11/9)\cdot (2^jp)\leq 2\cdot (11/9)\cdot 1/8<1/2$. Therefore, we can apply Lemma~\ref{lem:quantile_gap_stability} with 
\begin{align*}
    (\alpha_1,\alpha_2,\beta_1,\beta_2)=(9/11,11/9, 2\cdot 9/11,2\cdot 11/9)
\end{align*}
to claim that
\begin{align*}
   \widehat{\ell}^-_j&=\mu+q(U'_{2^{j+1}k})-\left(\mu+q(U'_{2^jk})\right)\\
   &=q(S^-_{2^{j+1}k}/S_{m+1})-q(S^-_{2^jk}/S_{m+1})\in [\ell(2^jp)/4,4\ell(2^jp)].
\end{align*}
The same bounds hold for $\widehat{\ell}_j^+$ alike, and thus also for their average $\widehat{\ell}_j$. By the definition of $\widehat{\lambda}_{0:J_p}$, we have
\begin{align*}
    \widehat{\lambda}_j=\frac{2^j/\widehat{\ell}_j^2}{\sum_{t=0}^{J_p}2^t/\widehat{\ell}_t^2}\in [1/256, 256]\cdot \frac{2^j/\ell^2(2^jp)}{\sum_{t=0}^{J_p}2^t/\ell(2^tp)^2}\subseteq [1/256,256]\cdot \Delta_f^2(p)\frac{2^j}{\ell^2(2^jp)},
\end{align*}
which concludes the proof.
\end{proof}

\subsection{Proof of Proposition~\ref{prop:estimator}}
\begin{proof}[Proof of Proposition~\ref{prop:estimator}]
First, we prove that the mean of $Y_\lambda$ is zero. Since $S^-_{2^jk}\eqd S^+_{2^jk}$ by symmetry, we have
\begin{align*}
    q(V^+_j)\eqd q(V^-_j),
\end{align*}
and thus
\begin{align*}
    Y_\lambda=\sum_{j=0}^{J_p}\lambda_j\frac{q(V^-_j)-q(V^+_j)}{2}
\end{align*}
is a symmetric random variable. Its expectation must exist finitely since its absolute value is bounded by $-\sum_j \lambda_j q(2^{j-1}p)<+\infty$ almost surely due to the truncation. Therefore, we have $\E[Y_\lambda]=0$ for any set of convex weights $\lambda_{0:J_p}$.

Combining Lemma~\ref{lem:truncated_mean_and_variance} and Lemma~\ref{lem:gap_and_weight_stable}, conditioned on any realization of $\cD'_m$ that satisfies $\cE'$, we have
\begin{align*}
    \max_{i\in[m+1]}L_i\leq \frac{8}{k}\sum_{j=0}^{J_p}\Delta_f^2(p)\frac{256(2^j)}{\ell^2(2^jp)}\cdot \frac{\ell(2^jp)}{2^j}&\leq \frac{2^{11}\Delta_f^2(p)}{k}\sqrt{\sum_{j=0}^{J_p}\frac{2^j}{\ell^2(2^jp)}}\cdot\sqrt{\sum_{j=0}^{J_p}\frac{1}{2^j}}\\
    &= \frac{2^{12}\Delta_f(p)}{k},
\end{align*}
where the second inequality is by Cauchy-Schwarz. Also,
\begin{align*}
    \sum_{i=1}^{m+1}L_i^2\leq \frac{1024}{k}\sum_{j=0}^{J_p}\frac{\ell^2(2^jp)}{2^j}\cdot 256^2\Delta_f^4(p)\frac{(2^j)^2}{\ell^4(2^jp)}=\frac{2^{26}\Delta_f^2(p)}{k}.
\end{align*}
We then use Lemma~\ref{lem:efron_stein_coordinatewise_lipschitz} to claim that, when $\cD'_m$ satisfies $\cE'$,
\begin{align*}
    \Prob_{\cD_m}\left(\abs{Y_{\widehat{\lambda}}}> 2\sqrt{\log(6/\delta)\frac{2^{26}\Delta_f^2(p)}{k}}+\log(6/\delta)\frac{2^{12}\Delta_f(p)}{k}\mid\cD'_m\right)\leq \delta/3.
\end{align*}
With our choice that $k=\ceil{720\log(1/\delta)}$, the right-hand side inside the bracket is upper bounded by
\begin{align*}
    \left(2^{13}\sqrt{\frac{\log(6/\delta)}{720\log(1/\delta)}}+2^{12}\frac{\log(6/\delta)}{720\log(1/\delta)}\right)\Delta_f(p)\leq 850\Delta_f(p).
\end{align*}
Recall that $Y_{\widehat{\lambda}}=\widehat{\mu}_\delta-\mu$ almost surely whenever $\cD_m$ satisfies $\cE$. We have
\begin{align*}
    \Prob_{\cD_m,\cD'_m}\left(\abs{\widehat{\mu}_\delta-\mu}>850\Delta_f(p)\right)&\leq \Prob_{\cD_m,\cD'_m}\left(Y=\widehat{\mu}_\delta,\,\abs{Y-\mu}>850\Delta_f(p)\right)+\Prob_{\cD_m,\cD'_m}\left(Y\neq \widehat{\mu}_\delta\right)\\
    &=\Prob_{\cD_m,\cD'_m}\left(Y=\widehat{\mu}_\delta,\,\abs{Y-\mu}>850\Delta_f(p)\right)+\Prob_{\cD_m}\left(\cE^\complement\right)\\
    &\leq \Prob_{\cD_m,\cD'_m}\left(\cE',\,\abs{Y-\mu}>850\Delta_f(p)\right)+\Prob_{\cD'_m}\left((\cE')^\complement\right)+\Prob_{\cD_m}\left(\cE^\complement\right)\\
    &= \E_{\cD'_m}\left[\Prob_{\cD_m}\left(\abs{Y-\mu}>2^{15}\Delta_f(p)\mid \cD'_m\right)\cdot \indi\{ \cE'(\cD'_m)\}\right]\nonumber\\
    &\ +\Prob_{\cD'_m}\left((\cE')^\complement\right)+\Prob_{\cD_m}\left(\cE^\complement\right)\\
    &\leq \delta/3+\delta/3+\delta/3=\delta.
\end{align*}
The desired conclusion is thus proved.
\end{proof}

\section{Other Proofs}
\subsection{Proof of Claim~\ref{clm:h2_log}}\label{sec:proof_h2_log}
\begin{proof}[Proof of Claim~\ref{clm:h2_log}]
We first prove the lower bound for the original Hellinger divergence. For $f(x)=(1-\abs{x})_+$, we have, by symmetry and the boundedness of the support, that
\begin{align*}
    \HL^2\left(f_0,f_{2r}\right)&=\frac{1}{2}\int_\bbR\left(\sqrt{f(x)}-\sqrt{f(x-2r)}\right)^2\dif x\\
    &=\int_{-1}^{r} \left(\sqrt{f(x)}-\sqrt{f(x-2r)}\right)^2\dif x\\
    &=\int_{-1}^{2r-1}\left(\sqrt{f(x)}-\sqrt{f(x-2r)}\right)^2\dif x+\int_{2r-1}^{r}\left(\sqrt{f(x)}-\sqrt{f(x-2r)}\right)^2\dif x.
\end{align*}
Since $f(x-2r)\equiv 0$ for $x\in [-1,2r-1]$, we have
\begin{align*}
    \int_{-1}^{2r-1}\left(\sqrt{f(x)}-\sqrt{f(x-2r)}\right)^2\dif x&=\int_{-1}^{2r-1}f(x)\dif x=\int_{-1}^{2r-1}(1+x)\dif x=2r^2.
\end{align*}
For the second integral, we have
\begin{align*}
\int_{2r-1}^{r}\left(\sqrt{f(x)}-\sqrt{f(x-2r)}\right)^2\dif x
    &\geq \int_{2r-1}^{0}\left(\sqrt{1+x}-\sqrt{1+x-2r}\right)^2\dif x\\
    &= \int_{0}^{1-2r}\left(\sqrt{x+2r}-\sqrt{x}\right)^2\dif x\\
    &=\int_{0}^{1-2r}\frac{4r^2}{\left(\sqrt{x+2r}+\sqrt{x}\right)^2}.
\end{align*}
Note that when $x\in [0,1-2r]$, 
\begin{align*}
    \left(\sqrt{x+2r}+\sqrt{x}\right)^2=2x+2r+2\sqrt{x(x+2r)}\leq 2x+2r+x+(x+2r)=4x+4r.
\end{align*}
Therefore, we can continue to write
\begin{align*}
    \int_{2r-1}^{r}\left(\sqrt{f(x)}-\sqrt{f(x-2r)}\right)^2\dif x\geq r^2\int_0^{1-2r}\frac{1}{x+r}\dif x=r^2\log\left(\frac{1-r}{r}\right).
\end{align*}
Combining the two pieces, we have
\begin{align*}
    \HL^2\left(f_0,f_{2r}\right)\geq 2r^2+r^2\log\left(\frac{1-r}{r}\right)\geq r^2\log(e/r),\quad 0\leq r\leq \frac{1}{2}.
\end{align*}

We then move on to the upper bound on the reduced Hellinger divergence. We first show that it suffices to consider intervals with $a=-\infty$, that is,
\begin{align*}
    \sup_{-\infty\leq a\leq b\leq +\infty}\HL^2\left(\cT_{[a,b]}\circ f_0,\, \cT_{[a,b]}\circ f_{2r}\right)&=\sup_{-\infty\leq b\leq +\infty} \HL^2\left(\cT_{[-\infty,b]}\circ f_0,\, \cT_{[-\infty,b]}\circ f_{2r}\right)\\
    &=\sup_{-\infty\leq b\leq +\infty} \HL^2\left(\mathrm{Bern}(F_0(b)),\, \mathrm{Bern}(F_{2r}(b))\right).
\end{align*}
To see this, define
\begin{align*}
    L(x)=\frac{f_{2r}(x)}{f_0(x)}=\begin{dcases}
        0,\quad &-1\leq x<2r-1;\\
        \frac{1+x-2r}{1+x},\quad &2r-1\leq x<0;\\
        \frac{1+x-2r}{1-x},\quad &0\leq x<2r;\\
        \frac{1+2r-x}{1-x},\quad &2r\leq x<1;\\
        +\infty,\quad &1\leq x\leq 1+2r.
    \end{dcases}
\end{align*}
Then, $L$ is non-decreasing on its domain $[-1,1+2r]$.
Let
\begin{align*}
    D(u,v)\define \HL^2\left(\mathrm{Bern}(u),\mathrm{Bern}(v)\right)=1-\sqrt{uv}-\sqrt{(1-u)(1-v)},\quad 0\leq u,v\leq 1.
\end{align*}
Let $I=[a,b]$ be a possibly infinite interval and set $u=F_0(I)$ and $v=F_{2r}(I)$. We may assume $v\leq u$ without loss of generality since we can reflect $I$ w.r.t. $r$ otherwise. Suppose first that $0<u<1$. Then, there exists a unique $t\in (-1,1)$ such that $F_0(t)=u$. Set $I'=[-\infty,t]$. We have $F_0(I')=F_0(I)=u$. We claim that $F_{2r}(I')\leq F_{2r}(I)=v$. Since $L$ is non-decreasing, we have
\begin{align*}
    f_{2r}(x)\geq L(t)f_0(x),\ x\geq t;\quad f_{2r}(x)\leq L(t)f_0(x),\ x\leq t.
\end{align*}
Therefore,
\begin{align*}
    F_{2r}(I)-F_{2r}(I')=F_{2r}(I\backslash I')-F_{2r}(I'\backslash I)&\geq L(t)\left(F_0(I\backslash I')-F_0(I'\backslash I)\right)\\
    &=L(t)\left(F_0(I)-F_0(I')\right)=0.
\end{align*}
This proves that $F_{2r}(I')\leq F_{2r}(I)=v$. Since $D(u,\cdot)$ is non-increasing on $[0,u]$, we have
\begin{align*}
    D(F_0(I),F_{2r}(I))=D(F_0(I'),F_{2r}(I))\leq D(F_0(I'),F_{2r}(I')).
\end{align*}
The same inequality holds automatically in the $u=0$ case because in this case we have $v\leq u=0$. Consequently, the supremum over $-\infty\leq a<b\leq \infty$ must be attained by a half-line, that is, either $a=-\infty$ or $b=+\infty$. By symmetry, it suffices to consider the case of $a=-\infty$.

Note that
\begin{align*}
    F(x)=\begin{dcases}
        0,\quad &x\leq -1;\\
        \frac{(1+x)^2}{2},\quad &-1\leq x<0;\\
        1-\frac{(1-x)^2}{2},\quad &0\leq x<1;\\
        1,\quad &x>1.
    \end{dcases}
\end{align*}
Define $G(x)=\arcsin\sqrt{F(x)}$. Then, $G'(x)=f(x)/(2\sqrt{F(x)(1-F(x))})$ by the chain rule. Specifically, we have
\begin{align*}
    G'(x)=\frac{1}{\sqrt{2-(1+x)^2}},\ -1<x<0;\quad G'(x)=\frac{1}{2-(1-x)^2},\ 0<x<1.
\end{align*}
Thus, uniformly on $(-1,1)$, we have $G'(x)\in (0,1)$. Moreover, $G(x)$ is continuous on $\bbR$ and constant on $(-\infty,-1]$ and $[1,+\infty)$. Consequently, $G(x)$ is $1$-Lipschitz on $\bbR$. Then, for every $b\in (-+\infty)$, 
\begin{align*}
    \HL^2\left(\mathrm{Bern}(F_0(b)),\, \mathrm{Bern}(F_{2r}(b))\right)&=1-\sqrt{F_0(b)F_{2r}(b)}-\sqrt{(1-F_0(b))(1-F_{2r}(b))}\\
    &=1-\cos\left(G(t)-G(t-2r)\right)\\
    &\leq \frac{1}{2}\left(G(t)-G(t-2r)\right)^2\\
    &\leq \frac{1}{2}\left(t-(t-2r)\right)^2=2r^2,
\end{align*}
where we first use $1-\cos(x)\leq x^2/2$ for $x\in[-1,1]$ and then the $1$-Lipschitzness of $G$. The desired conclusion is thus proved.
\end{proof}

\subsection{Proof of the Optimality of Parametric MLE}\label{sec:proof_mle}
\begin{Claim}\label{clm:mle}
For any $\delta\in(0,1/3]$ and $f\in\cF_\mathsf{SLC}$, and the parametric MLE estimator $\widehat{\mu}_f^{(\mathsf{MLE})}$ defined by \eqref{eq:mle}, we have
\begin{align*}
    \inf_{\mu\in\bbR}\, \Prob_{X_{1:n}\iid f_\mu}\left(\abs{\widehat{\mu}_f^{(\mathsf{MLE})}-\mu}\leq \omega_f\Big(\frac{C\log(1/\delta)}{n}\Big)\right)\geq 1-\delta\quad \textnormal{whenever}\ n\geq C'\log(1/\delta),
\end{align*}
where $C$ and $C'$ are universal absolute constants. 
\end{Claim}
\begin{proof}
Let
\begin{align*}
    L_n(\theta)=-\sum_{i=1}^n \log f(X_i-\theta),\quad \theta\in\bbR.
\end{align*}
Then, $L_n(\theta)$ is a convex function in $\theta$ due to the log-concavity of $f$. For any $r>0$, we have
\begin{align*}
    \left\{\abs{\widehat{\mu}_f^{(\mathsf{MLE})}-\mu}\leq r\right\}\supseteq \left\{L_n(\mu)\leq \inf_{\theta:\abs{\theta-\mu}\geq r}L_n(\theta)\right\}.
\end{align*}
Moreover, we claim that
\begin{align*}
    \left\{L_n(\mu)\leq \inf_{\theta:\abs{\theta-\mu}\geq r}L_n(\theta)\right\}\supseteq \Big\{L_n(\mu)\leq \min\left\{L_n(\mu+r),L_n(\mu-r)\right\}\Big\}.
\end{align*}
We prove this by contradiction. Suppose both
\begin{align*}
    L_n(\mu)> \inf_{\theta:\abs{\theta-\mu}\geq r}L_n(\theta),\quad L_n(\mu)\leq \min\left\{L_n(\mu+r),L_n(\mu-r)\right\}
\end{align*}
hold at the same time. Then, there exists $\theta$ with $\abs{\theta-\mu}>r$ such that
\begin{align*}
    L_n(\theta)<L_n(\mu)\leq \min\left\{L_n(\mu+r),L_n(\mu-r)\right\}.
\end{align*}
Without loss of generality, we assume $\theta>\mu+r$. Then, we notice that
\begin{align*}
    \mu<\mu+r<\theta,\quad \textnormal{but}\ L_n(\theta)<L_n(\mu)\leq L_n(\mu+r),
\end{align*}
which contradicts the convexity of $L_n$. Therefore, we have
\begin{align*}
    &\ \Prob_{X_{1:n}\iid f_\mu}\left(\abs{\widehat{\mu}_f^{(\mathsf{MLE})}-\mu}\leq r\right)\\
    &\geq  \Prob_{X_{1:n}\iid f_\mu}\left(L_n(\mu)\leq \min\left\{L_n(\mu-r),L_n(\mu+r)\right\}\right)\\
    &\geq \min\left\{\Prob_{X_{1:n}\iid f_\mu}\left(L_n(\mu)\leq L_n(\mu-r)\right),\, \Prob_{X_{1:n}\iid f_\mu}\left(L_n(\mu)\leq L_n(\mu+r)\right)\right\}\\
    &\geq 1-\Prob_{X_{1:n}\iid f_\mu}\left(\sum_{i=1}^n\log\left(\frac{f(X_i-\mu-r)}{f(X_i-\mu)}\right)\geq 0\right)-\Prob_{X_{1:n}\iid f_\mu}\left(\sum_{i=1}^n\log\left(\frac{f(X_i-\mu+r)}{f(X_i-\mu)}\right)\geq 0\right)\\
    &=1-2\Prob_{X_{1:n}\iid f_\mu}\left(\exp\left[\frac{1}{2}\sum_{i=1}^n\log\left(\frac{f(X_i-\mu-r)}{f(X_i-\mu)}\right)\right]\geq 1\right)\\
    &\geq 1- 2\E_{X_{1:n}\iid f_\mu}\left[\exp\left[\frac{1}{2}\sum_{i=1}^n\log\left(\frac{f(X_i-\mu-r)}{f(X_i-\mu)}\right)\right]\right]\\
    &=1-2\left(\E_{X\sim  f_\mu}\left[\sqrt{f(X-\mu-r)/f(X-\mu)}\right]\right)^n\\
    &=1-2\left(1-\HL^2\left(f_0,f_r\right)\right)^n\\
    &\geq 1-2\exp\left(-n\HL^2\left(f_0,f_{r}\right)\right),
\end{align*}
where we have used the Chernoff trick. Now, choose $r>0$ as
\begin{align*}
    \HL^2\left(f_0,f_r\right)=\frac{\log(2/\delta)}{n}.
\end{align*}
This choice of $r$ensures that the above probability is at least $1-\delta$ as desired. Meanwhile, such an $r>0$ exists finitely for any $\delta\in(0,1/3]$ as long as $n\geq C'\log(1/\delta)$ for a large enough constant $C'>0$. Moreover, for this choice of $r$, we have
\begin{align*}
     \HL^2\left(f_0,f_{2r}\right)\leq \left(\HL\left(f_0,f_r\right)+\HL\left(f_r,f_{2r}\right)\right)^2=4\HL^2\left(f_0,f_r\right)=\frac{4\log(2/\delta)}{n}
\end{align*}
since $\HL(\cdot,\cdot)$ is a distance. Therefore, for $\delta\in (0,1/3]$ and $n\geq C'\log(1/\delta)$ with a large enough universal $C'>0$, we will have
\begin{align*}
    r\leq \omega_f\left(\frac{C\log(1/\delta)}{n}\right)
\end{align*}
for another universal $C>0$. The desired conclusion then follows.
\end{proof}

\section{Auxiliary Results}\label{sec:aux}
\begin{Lemma}\label{lem:toeplitz_sandwich}
Let $N\in\bbZ_+$ and $\rho\in(0,1)$. Then,
\begin{align*}
    \frac{1-\rho}{1+\rho}I_N\preceq \left(\rho^{\abs{j-l}}\right)_{j,l=1}^N\preceq \frac{1+\rho}{1-\rho}I_N.
\end{align*}
\end{Lemma}
\begin{proof}
Denote the matrix $(\rho^{\abs{j-l}})_{j,l=1}^N$ by $A_N(\rho)$. Fix any $x\in\bbR^N$. Define the trigonometric polynomial
\begin{align*}
    X(\theta)=\sum_{j=1}^N x_j e^{\im j\theta},\quad \theta\in[-\pi,\pi].
\end{align*}
Then, we have
\begin{align*}
    \abs{X(\theta)}^2=\sum_{j,l=1}^Nx_jx_l e^{\im(j-l)\theta},
\end{align*}
and by Parseval's identity,
\begin{align*}
    \frac{1}{2\pi}\int_{-\pi}^\pi\abs{X(\theta)}^2\dif \theta=\norm{x}_2^2.
\end{align*}
Consider the Poisson kernel
\begin{align*}
    P_\rho(\theta)\define \frac{1-\rho^2}{1-2\rho\cos(\theta)+\rho^2}=\sum_{h\in\bbZ}\rho^{\abs{h}}e^{\im h\theta}.
\end{align*}
Fubini's theorem then gives
\begin{align*}
    \frac{1}{2\pi}\int_{-\pi}^\pi P_\rho(\theta)\abs{X(\theta)}^2\dif \theta=\sum_{h\in\bbZ}\sum_{j,l=1}^N\rho^{\abs{h}}x_jx_l\frac{1}{2\pi}\int_{-\pi}^\pi e^{\im(h+j-l)\theta}\dif \theta=\sum_{j,l=1}^N\rho^{\abs{j-l}}x_jx_l=x^\top A_N(\rho)x.
\end{align*}
Meanwhile, since $\cos(\theta)\in[-1,1]$, we have
\begin{align*}
   \frac{1-\rho}{1+\rho}\leq P_\rho(\theta)\leq \frac{1+\rho}{1-\rho},
\end{align*}
and thus
\begin{align*}
    x^\top A_N(\rho)x=\frac{1}{2\pi}\int_{-\pi}^\pi P_\rho(\theta)\abs{X(\theta)}^2\dif \theta&\in\left[\frac{1-\rho}{1+\rho}\frac{1}{2\pi}\int_{-\pi}^\pi \abs{X(\theta)}^2\dif \theta,\, \frac{1+\rho}{1-\rho}\frac{1}{2\pi}\int_{-\pi}^\pi \abs{X(\theta)}^2\dif \theta\right]\\
    &=\left[\frac{1-\rho}{1+\rho}\norm{x}_2^2,\,\frac{1+\rho}{1-\rho}\norm{x}_2^2\right],
\end{align*}
where the last line follows from Parseval's identity.
\end{proof}

\begin{Lemma}\label{lem:gamma_concentration}
    Let $N\in\bbZ_+$ and $G_N=\sum_{i=1}^N E_i$, where $E_{1:N}$ are independent $\mathrm{Exp}(1)$ random variables. Then, for every $\eta\in(0,1)$, we have
    \begin{align*}
        \Prob(G_N\geq (1+\eta)N)\leq \exp\left(-N(\eta-\log(1+\eta))\right),
    \end{align*}
    and
    \begin{align*}
        \Prob(G_N\leq (1-\eta)N)\leq \exp\left(-N(-\eta-\log(1-\eta))\right).
    \end{align*}
\end{Lemma}
\begin{proof}
For $E_i\sim \mathrm{Exp}(1)$ and $\theta<1$, we have
\begin{align*}
    \E[\exp(\theta E_i)]=\frac{1}{1-\theta}.
\end{align*}
By independence, we have
\begin{align*}
    \E[\exp(\theta G_N)]=\frac{1}{(1-\theta)^N}.
\end{align*}
By Markov's inequality and that $x\mapsto e^{\theta x}$ is increasing for $\theta\in(0,1)$,
\begin{align*}
    \Prob\left(G_N\geq (1+\eta)N\right)&=\Prob\left(\exp(\theta G_N)\geq \exp(\theta(1+\eta)N)\right)\\
    &\leq e^{-\theta(1+\eta)N}\E\left[\exp(\theta G_N)\right]\\
    &=\exp\left(-N(\theta(1+\eta)+\log(1-\theta))\right),\quad \theta\in(0,1).
\end{align*}
Setting $\theta=\eta/(1+\eta)\in (0,1)$ proves the first inequality. Similarly, since $x\mapsto e^{\theta x}$ is decreasing for $\theta<0$, we have
\begin{align*}
    \Prob\left(G_N\leq (1-\eta)N\right)&=\Prob\left(\exp(\theta G_N)\geq \exp(\theta(1-\eta)N)\right)\\
    &\leq e^{-\theta(1-\eta)N}\E\left[\exp(\theta G_N)\right]\\
    &=\exp\left(-N(\theta(1-\eta)+\log(1-\theta))\right),\quad \theta<0.
\end{align*}
Setting $\theta=-\eta/(1-\eta)<0$ proves the second inequality.
\end{proof}

\begin{Lemma}\label{lem:difference_hardy}
Let $a_0>0$ and $a_j=2^j a_0$. For any $N\in\bbZ_+$ and any real numbers $b_{0:N}$, 
\begin{align*}
    \sum_{j=0}^N\frac{b_j^2}{a_j}\leq 4\frac{b_0^2}{a_0}+12\sum_{j=1}^J\frac{(b_j-b_{j-1})^2}{a_j}.
\end{align*}
\end{Lemma}
\begin{proof}
Let $e_j=b_j-b_{j-1}$. The inequality $(x+y)^2\leq 3x^2/2+3b^2$ gives
    \begin{align*}
        \frac{b_j^2}{a_j}=\frac{(b_{j-1}+e_j)^2}{a_j}\leq \frac{3}{4}\frac{b_{j-1}^2}{a_{j-1}}+3\frac{e_j^2}{a_j}.
    \end{align*}
Let $S=\sum_{j=0}^N b_j^2/a_j$. Summing the preceding inequality over $j=1,2,\cdots,N$, we obtain
\begin{align*}
    S-\frac{b_0^2}{a_0}\leq  \sum_{j=1}^N \left(\frac{3}{4}\frac{b_{j-1}^2}{a_{j-1}}+3\frac{e_j^2}{a_j}\right)\leq \frac{3}{4} S+3\sum_{j=1}^N\frac{e_j^2}{a_j}.
\end{align*}
Rearranging the terms concludes the proof.
\end{proof}

\begin{Lemma}[Pr\'{e}kopa--Leindler \cite{prekopa1971logarithmic, leindler1972certain}]\label{lem:prekopa_leindler} Let $\lambda\in(0,1)$. Let $f,g,h:\bbR\to [0,+\infty)$ be nonnegative Lebesgue measurable functions on $\bbR$. Suppose that
\begin{align*}
    h((1-\lambda)x+\lambda y)\geq f^{1-\lambda}(x)g^\lambda(y),\quad x,y\in\bbR.
\end{align*}
Then,
\begin{align*}
    \int_\bbR h(x)\dif x\geq \left(\int_\bbR f(x)\dif x\right)^{1-\lambda}\left(\int_\bbR g(x)\dif x\right)^{\lambda}.
\end{align*}
\end{Lemma}

\end{sloppypar}
\end{document}